\title{Elliptic fibrations on toric $K3$ hypersurfaces and mirror symmetry derived from Fano polytopes}
\author{Tomonao Matsumura and Atsuhira Nagano}
\documentclass[10pt]{article}
\usepackage{mathrsfs,bm,graphics,amsfonts,amsthm,amssymb,amsmath,amscd,booktabs}
\usepackage{url}
\usepackage{longtable}
\usepackage{tikz}
\DeclareFontFamily{U}{mathx}{}
\DeclareFontShape{U}{mathx}{m}{n}{ <-> mathx10 }{}
\DeclareSymbolFont{mathx}{U}{mathx}{m}{n}
\DeclareFontSubstitution{U}{mathx}{m}{n}
\DeclareMathAccent{\widecheck}{0}{mathx}{"71}
\def\bigzerou{\smash{\lower1.7ex\hbox{\b 0}}}

\newtheorem{thm}{Theorem}[section]
\newtheorem{df}{Definition}[section]
\newtheorem{lem}{Lemma}[section]
\newtheorem{prop}{Proposition}[section]
\newtheorem{rem}{Remark}[section]

\newtheorem{cor}{Corollary}[section]

\newtheorem{conj}{Conjecture}[section]

\def\comment#1{{ }}
\makeatletter
  
      \@addtoreset{equation}{section}
      
\begin{document}
\maketitle

\begin{abstract}
We determine the N\'eron-Severi lattices of $K3$ hypersurfaces with large Picard number in toric three-folds derived from Fano polytopes.
On each $K3$ surface, 
we introduce a particular elliptic fibration.
In the proof of the main theorem,
we show that the N\'eron-Severi lattice of each $K3$ surface is generated by a general fibre, sections and appropriately selected components of the singular fibres of our elliptic fibration.
Our argument gives a certain proof of  the Dolgachev conjecture  for  Fano polytopes, which is a conjecture on mirror symmetry for  $K3$ surfaces.
\end{abstract}

\setlength{\baselineskip}{14 pt}

\section*{Introduction}
In this paper,
we  determine the N\'eron-Severi lattices and the transcendental lattices of $K3$ surfaces
which are given as hypersurfaces in toric three-folds derived from three-dimensional Fano polytopes.
These lattices are important in geometry and number theory.
Our argument  gives a  proof of a conjecture in mirror symmetry for $K3$ surfaces.
The proof of our main theorem is based on  precise studies of  appropriate Jacobian elliptic fibrations on our $K3$ surfaces.

Mirror symmetry is a phenomenon originally discovered by physicists  studying superstring theory.
It suggests  non-trivial relationships between geometric objects
and provides a lots of interesting problems in algebraic geometry.
Especially,
many mathematicians have intensively studied Calabi-Yau hypersurfaces in toric varieties coming from reflexive polytopes since the publication of \cite{B}.
Calabi-Yau hypersurfaces in toric three-folds are $K3$ surfaces.
There is a famous conjecture  on mirror symmetry for such toric $K3$ hypersurfaces by Dolgachev \cite{D}.
This conjecture is formulated in terms of the following lattices associated with $K3$ surfaces.
The $2$-cohomology group $H^2(S,\mathbb{Z})$ of a $K3$ surface $S$ gives an even unimodular lattice $L_{K3}$ of signature $(3,19).$ 
The N\'eron-Severi lattice ${\rm NS}(S)$ is a sublattice of $H^2(S,\mathbb{Z})$ defined as $H^2(S,\mathbb{Z}) \cap H^{1,1}(S)$.
We call $\rho={\rm rank}({\rm NS} (S))$ the Picard number of $S$.
Then, ${\rm NS}(S)$ is of signature $(1,\rho-1)$.
We will identify $H^2(S,\mathbb{Z})$ with the $2$-homology group $H_2(S,\mathbb{Z})$ by the Poincar\'e duality.
Thereupon, ${\rm NS}(S)$ is regarded as the sublattice of $H_2(S,\mathbb{Z})$ generated by divisors on $S$. 
The transcendental lattice ${\rm Tr}(S)$ is the orthogonal complement of ${\rm NS}(S)$ in $H_2(S,\mathbb{Z})$.
Here, arithmetic techniques of  even lattices are powerful tools to study the above lattices  (for example, see \cite{Ni}).
Thus, mirror symmetry for toric $K3$ hypersurfaces is an interesting research theme involving researchers from various fields.

Fano polytopes are special reflexive polytopes (see Definition \ref{dfFano}).
There exists a smooth Fano $n$-fold attached to each $n$-dimensional Fano polytope.
Since Fano polytopes have nice combinatorial properties,
many researchers have  studied them (for example, see  \cite{BToric}, \cite{BToric4},  \cite{WW}, \cite{Sa} and \cite{Ob}).
In particular, three-dimensional Fano polytopes are classified into $18$ types up to ${\rm GL}_3(\mathbb{Z})$-action.
The number of vertices of such a polytope is one of $4,5,6,7$ or $8$.
See Table 1.
Here, each column of the matrix in the table  stands for  the coordinates of each vertex of the corresponding  polytope.

{\small
\begin{longtable}{lll}
\caption{Fano polytopes $P_k$ and Lattices $L_k$}\\
\hline
  $k$ & $P_k$  & $L_k$ \vspace{1mm}  \\
  \hline
  \endhead
\vspace{1mm}$1$ & $ \left(
\begin{array}{cccccc}
1 & 0 & 0 & -1 \\
0 & 1 & 0 & -1  \\
0 & 0 & 1 & -1  \\
\end{array}\right)$ & $(4)$ \vspace{1mm}\\
\vspace{1mm}$2$ & $ \left(
\begin{array}{cccccc}
1 & 0 & 0 & -1 & 0  \\
0 & 1 & 0 & -1 & 0  \\
0 & 0 & 1 & 0 & -1  \\
\end{array}\right)$ &$\begin{pmatrix} 0 & 3 \\ 3 & 2 \end{pmatrix}$ \vspace{1mm}\\
\vspace{1mm}$3$ &
$ \left(
\begin{array}{cccccc}
1 & 0 & 0 & 0 & -1  \\
0 & 1 & 0 & 0 & -1  \\
0 & 0 & 1 & -1 & -1  \\
\end{array}\right)$ & $\begin{pmatrix} 2 & 2 \\ 2 & -2 \end{pmatrix}$ \vspace{1mm}\\
\vspace{1mm}$4$ &$ \left(
\begin{array}{cccccc}
1 & 0 & 0 & 0 & -1  \\
0 & 1 & 0 & 0 & -1  \\
0 & 0 & 1 & -1 & -2  \\
\end{array}\right)$ & 
$\begin{pmatrix} 2 & 1 \\ 1 & -2 \end{pmatrix}$ \vspace{1mm}\\
\vspace{1mm}$5$ &
$ \left(
\begin{array}{cccccc}
1 & 0 & 0 & -1 & -1  \\
0 & 1 & 0 & -1 & -1  \\
0 & 0 & 1 & 0 & -1  \\
\end{array}\right)$ & $\begin{pmatrix} 0 & 3 \\ 3 & -2 \end{pmatrix}$\vspace{1mm} \\
\vspace{1mm}$6$ &$ \left(
\begin{array}{cccccc}
1 & 0 & 0 & 0 & 0 & -1 \\
0 & 1 & 0 & 0 & -1 & 0 \\
0 & 0 & 1 & -1 & 0 & 0 \\
\end{array}\right)$ & $\left(\begin{array}{ccc}
0 & 2 & 2 \\
2 & 0 & 2 \\
2 & 2 & 0 \\
\end{array}\right) $\vspace{1mm} \\
\vspace{1mm}$7$ &$\left(
\begin{array}{cccccc}
1 & 0 & 0 & 0 & 0 & -1 \\
0 & 1 & 0 & 0 & -1 & 0 \\
0 & 0 & 1 & -1 & -1 & -1 \\
\end{array}\right) $ & $\left(\begin{array}{ccc}
-2 & 1 & 1 \\
1 & 0 & 2 \\
1 & 2 & 0 \\
\end{array}\right)$ \vspace{1mm}\\
\vspace{1mm}$8$ &$\left(
\begin{array}{cccccc}
1 & 0 & 0 & 0 & 0 & -1 \\
0 & 1 & 0 & 0 & -1 & 0 \\
0 & 0 & 1 & -1 & 1 & -1 \\
\end{array}\right)$ &$\left(\begin{array}{ccc}
-2 & 1 & 3 \\
1 & 0 & 2 \\
3 & 2 & 0 \\
\end{array}\right) $ \vspace{1mm}\\
\vspace{1mm}$9$ & $\left(
\begin{array}{cccccc}
1 & 0 & 0 & 0 & 0 & -1 \\
0 & 1 & 0 & 0 & -1 & -1 \\
0 & 0 & 1 & -1 & 0 & 0 \\
\end{array}\right)$ & $\left(\begin{array}{ccc}
0 & 1 & 2 \\
1 & -2 & 2 \\
2 & 2 & 0 \\
\end{array}\right)$ \vspace{1mm}\\
\vspace{1mm}$10$ & $\left(
\begin{array}{cccccc}
1 & 0 & 0 & 0 & 0 & -1 \\
0 & 1 & 0 & 0 & -1 & -1 \\
0 & 0 & 1 & -1 & -1 & -1 \\
\end{array}\right)$ & $\left(\begin{array}{ccc}
-2 & 1 & 1 \\
1 & -2 & 2 \\
1 & 2 & 0 \\
\end{array}\right)$ \vspace{1mm}\\
\vspace{1mm}$11$ & $\left(
\begin{array}{cccccc}
1 & 0 & 0 & 0 & 0 & -1 \\
0 & 1 & 0 & 0 & 1 & -1 \\
0 & 0 & 1 & -1 & -1 & -1 \\
\end{array}\right)$ & $\left(\begin{array}{ccc}
-2 & 1 & 1 \\
1 & -2 & 1 \\
1 & 1 & 2 \\
\end{array}\right)$ \vspace{1mm}\\
\vspace{1mm}$12$ & $\left(
\begin{array}{cccccc}
1 & 0 & 0 & 0 & 0 & -1 \\
0 & 1 & 0 & 0 & 1 & -1 \\
0 & 0 & 1 & -1 & -1 & 0 \\
\end{array}\right)$ & $\left(\begin{array}{ccc}
-2 & 2 & 2 \\
2 & -2 & 1 \\
2 & 1 & 2 \\
\end{array}\right)$ \vspace{1mm}\\
\vspace{1mm}$13$ & $\left(
\begin{array}{ccccccc}
1 & 0 & 0 & 0 & 0 & -1 & -1\\
0 & 1 & 0 & 0 & -1 & 0 & -1\\
0 & 0 & 1 & -1 & 0 & 0 & 0\\
\end{array}\right)$ & $\left(\begin{array}{cccc}
0 & 1 & 1 & 1 \\
1 & -2 & 0 & 2 \\
1 & 0 & -2 & 2 \\
1 & 2 & 2 & -2 \\
\end{array}\right)$ \vspace{1mm}\\
\vspace{1mm}$14$ & $\left(
\begin{array}{ccccccc}
1 & 0 & 0 & -1 & 0 & -1 & -1\\
0 & 1 & 0 & 0 & -1 & 0 & -1\\
0 & 0 & 1 & -1 & 0 & 0 & 0\\
\end{array}\right)$ & $\left(\begin{array}{cccc}
0 & 1 & 1 & 1 \\
1 & -2 & 0 & 2 \\
1 & 0 & -2 & 1 \\
1 & 2 & 1 & -2 \\
\end{array}\right) $ \vspace{1mm}\\
\vspace{1mm}$15$ & $\left(
\begin{array}{ccccccc}
1 & 0 & 0 & 1 & 0 & -1 & -1\\
0 & 1 & 0 & 0 & -1 & 0 & -1\\
0 & 0 & 1 & -1 & 0 & 0 & 0\\
\end{array}\right)$ &$\left(\begin{array}{cccc}
0 & 1 & 1 & 1 \\
1 & -2 & 0 & 2 \\
1 & 0 & -2 & 3 \\
1 & 2 & 3 & -2 \\
\end{array}\right)$ \vspace{1mm}\\
\vspace{1mm}$16$ & $\left(
\begin{array}{ccccccc}
1 & 0 & 0 & -1 & 0 & -1 & -1\\
0 & 1 & 0 & -1 & -1 & 0 & -1\\
0 & 0 & 1 & -1 & 0 & 0 & 0\\
\end{array}\right)$ & $\left(\begin{array}{cccc}
0 & 1 & 1 & 1 \\
1 & -2 & 0 & 1 \\
1 & 0 & -2 & 1 \\
1 & 1 & 1 & -2 \\
\end{array}\right)$ \vspace{1mm}\\
\vspace{1mm}$17$ & $\left(
\begin{array}{cccccccc}
1 & 0 & 0 & 0 & 0 & -1 & 1 & -1 \\
0 & 1 & 0 & 0 & -1 & 0 & -1 & 1 \\
0 & 0 & 1 & -1 & 0 & 0 & 0 & 0 \\
\end{array}\right)$ & $\left(\begin{array}{ccccc}
0 & 1 & 1 & 1 & 1 \\
1 & -2 & 2 & 2 & 0 \\
1 & 2 & -2 & 0 & 2 \\
1 & 2 & 0 & -2 & 0 \\
1 & 0 & 2 & 0 & -2 \\
\end{array}\right)$ \vspace{1mm}\\
$18$ & $\left(
\begin{array}{cccccccc}
1 & 0 & 0 & 1 & 0 & -1 & 1 & -1 \\
0 & 1 & 0 & -1 & -1 & 0 & -1 & 1 \\
0 & 0 & 1 & -1 & 0 & 0 & 0 & 0 \\
\end{array}\right)$ & $\left(\begin{array}{ccccc}
0 & 1 & 1 & 1 & 1 \\
1 & -2 & 2 & 1 & 0 \\
1 & 2 & -2 & 0 & 3 \\
1 & 1 & 0 & -2 & 0 \\
1 & 0 & 3 & 0 & -2 \\
\end{array}\right)$ \vspace{1mm}\\
\hline
\end{longtable} 
}

From a three-dimensional Fano polytope $P$, we obtain another polytope $P^\circ$, which is called the polar dual.
The duality between $P$ and $P^\circ$ derives two different families of $K3$ surfaces.
Namely, we have the family of $S_{P^\circ}$ with small Picard number 
and the family of $S_P$ with large Picard number.
Let $S_P$ and $S_{P^\circ}$ are very general members.
Dolgachev \cite{D} conjectures that $(S_P , S_{P^\circ})$ gives a mirror pair.
This means that
 there is  a  relationship between the lattices of $S_P$ and those of $S_{P^\circ}$ (for detail, see Section 1).
Here,
the $K3$ surface $S_{P^\circ}$ is a hypersurface in a smooth toric Fano three-fold.
By applying a result in algebraic geometry (see Proposition \ref{PropMoishezon}),
one can determine ${\rm NS}(S_{P^\circ})$ systematically.
Indeed, Koike \cite{K2} and Mase \cite{Mase} determine the N\'eron-Severi lattice  of $S_{P^\circ}$ for every three-dimensional Fano polytope $P$ (see Proposition \ref{PropLatticeDual}).
On the other hand,
to the best of the authors' knowledge,
it is not easy to determine the N\'eron-Severi lattice of $S_P$ for every $P$,
because its Picard number is large and its ambient toric variety has singularities.

The  subjects of the present paper are the $K3$ surfaces $S_P$.
Here, we note that 
Narumiya and Shiga \cite{NS} determine the lattice structure of $S_P$
if the number of vertices of  $P$ is $4$.
Also, if the number of vertices of $P$ is $5$,
the second author \cite{Na} determines the lattice structure of $S_P$ (see also \cite{I} and \cite{IIT}).
In this paper, 
we will determine the  lattice structure of  $S_{P}$ for every three-dimensional Fano polytope $P$ with $6,7$ or $8$ vertices.
Our main  theorem is given as follows.

\begin{thm}\label{ThmMain}
Let $P_k$ ($k\in \{1,\ldots,18\}$) be the Fano polytope in Table 1.
Then, the intersection matrix of the transcendental lattice ${\rm Tr}(S_{P_k})$ is isometric to $U\oplus L_k$,
where $U$ is the hyperbolic lattice of rank two
and $L_k$ is the lattice in Table 1.
\end{thm}

Indeed, 
we will determine the structure of ${\rm NS} (S_{P_k})$ $(k\in \{6,\ldots,18\})$ in Theorem \ref{ThmNSEvident}.
Theorem \ref{ThmMain} immediately follows from  Theorem \ref{ThmNSEvident} (see Corollary \ref{CorTr}).
Also, this main theorem establishes that  the Dolgachev conjecture holds for Fano polytopes.

In our proof, it is important to study divisors on  $S_P$ with large Picard number.
For the cases of $S_P$,
unlike the cases of $S_{P^\circ}$,
it does not appears to be enough to apply  techniques of toric divisors.
Therefore,
our mathematical argument is much different to those of \cite{K2} or \cite{Mase}.
We will explicitly take a divisor  $F$ on each $S_P$ whose self intersection number is zero.  
Then, 
we obtain an elliptic fibration $\pi_P: S_P \rightarrow \mathbb{P}^1 (\mathbb{C})$
with a general fibre $F$.
Each  fibration $\pi_P$ has  sections $\mathbb{P}^1(\mathbb{C}) \rightarrow S_P$.
Such a fibration  is called a Jacobian elliptic fibration.
We will make full use of arithmetic properties of even lattices and  Mordell-Weil groups  in order to determine ${\rm NS}(S_P)$ in Section 4 and 5.

As a matter of fact,
our proof of the main theorem is heavily due to nice properties of the fibration $\pi_P$.
Let us summarize our argument.
For each $P$, take a very general member $S_P$ of the family.
We will introduce  a sublattice $E_P$ of ${\rm NS}(S_P)$ generated by a general fibre, sections and  appropriately selected irreducible components of the singular fibres of $\pi_P$ (see Table 5 and 6). 
The lattice $E_P$ will be called the evident lattice attached to $\pi_P.$
The absolute value $|\det (E_P)|$  is equal to $|\det({\rm NS}(S_{P^\circ}))|$,
which is already calculated in \cite{K2} or \cite{Mase}.
Moreover,
our lattice $E_P$ defines  a period mapping for the family of $S_P$  as in Section 3.
We will see  the Picard number of $S_P$ is equal to the rank of $E_P$ (Theorem \ref{ThmPic}). 
We will study the invariant of the even lattice $E_P$ in Section 4.
Furthermore, in Section 5, we will determine the structure of the Mordell-Weil group ${\rm MW}(\pi_P,O)$ associated with $\pi_P$ for each Fano polytope $P$.
Here, $O $ is the section of $\pi_P$ corresponding to the identity element of the group.
Then, we will prove that $E_P$ is equal to ${\rm NS}(S_P)$ (Theorem \ref{ThmNSEvident}) and that the Dolgachev conjecture holds for $P$ (Corollary \ref{CorTr}).
We note that each $S_P$  has  several Jacobian elliptic fibrations.
However,
not every Jacobian elliptic fibration is effective to determine the lattice structure in question.
Our fibration $\pi_P$ initially  appears in \cite{M1} and \cite{M2}.
It  has profitable properties for our purpose.

Here, we mention  Rohsiepe's work in 2004.
In the preprint \cite{R}, he argued that he resolved the Dolgachev conjecture affirmatively by  computer calculations.
In particular, he wrote down an outline of an algorithm for his calculations.
However, unfortunately, 
it is difficult for us to verify his mathematical argument accurately as of the writing of the preprint.  
To the best of the authors' knowledge, the details are obscured in the coarse description of the algorithm in \cite{R} and  a website where he uploaded data of his calculations does not seem to have been working.

On the other hand,
our  proof of Theorem \ref{ThmMain} is based on detailed investigations of  the particular  elliptic fibrations.
They enable us to study the period mappings and the moduli spaces of the marked $K3$ surfaces polarized by our evident lattices via an argument similar to \cite{NaP} Section 2.
The authors expect that $K3 $ surfaces coming from Fano polytopes  are interesting objects of research with various properties.
Indeed, the second author proved in  \cite{NaC} that the inverse correspondence of the period mapping for the family of $K3$ surfaces $S_{P_4}$, which is  derived from the Fano polytope $P_4$ with $5$ vertices, gives a tuple of Hilbert modular forms of two variables. 
Thus,  our results are connected to future applications of  $K3$ surfaces with large Picard number.

\section{Mirror pairs of toric $K3$ hypersurfaces from Fano polytopes}

\subsection{Preliminaries and notations}

We start this paper with surveying the construction of toric varieties from reflexive polytopes.
For detailed arguments or proofs, see  \cite{B}, \cite{Od} and \cite{CLS}.

Let $M$ be a free abelian group of rank $n$.
We identify $M$ with $\mathbb{Z}^n.$
Let $N$ be the dual of $M$:
$N={\rm Hom}(M,\mathbb{Z}).$
We set $M_\mathbb{R}=M\otimes_\mathbb{Z} \mathbb{R}$ and $N_\mathbb{R}=N\otimes_\mathbb{Z} \mathbb{R}.$
The non-degenerate pairing $N_\mathbb{R} \times M_\mathbb{R}\rightarrow \mathbb{R}$ is denoted by
$\langle \nu,m\rangle$ for $\nu\in N_\mathbb{R}$ and $m\in M_\mathbb{R}$.
In this paper,
$m\in M_\mathbb{R}$ ($n\in N_\mathbb{R}$, resp.) will be usually regarded as an $n$-component column (row, resp.) vector 
and
$\langle \nu,m\rangle$ is regarded as the canonical inner product of $\mathbb{R}^n$.
Let $\mathbb{T}$ be a torus: $\mathbb{T} \simeq (\mathbb{C}^\times)^n$.
For $m=(m_1,\ldots, m_n)\in M$,
let $\chi^m$ be a character $\mathbb{T} \rightarrow \mathbb{C}^\times$ defined by
$\chi^m (t_1,\ldots, t_n) = t_1^{m_1}\cdots t_n^{m_n}$.

Let  $P \subset M_\mathbb{R}$ be an  $n$-dimensional  convex lattice polytope.
Here, a lattice polytope means a polytope whose vertices are integral.
Let $\mathfrak{F}$ be a face of $P$ and $\mathfrak{F}_0$ be a facet which contains $\mathfrak{F}$.
We can take the unique primitive vector $n_{\mathfrak{F}_0} \in N$ which gives the normal vector of $\mathfrak{F}_0$
 pointing towards the interior of $P$.
Let $\sigma_{\mathfrak{F}}\subset N_\mathbb{R}$ be the cone generated by $n_{\mathfrak{F}_0}$ for all facets  $\mathfrak{F}_0$ satisfying   $\mathfrak{F} \subset \mathfrak{F}_0.$
Then,
$
\Sigma(P)=\{\sigma_\mathfrak{F} \mid \mathfrak{F} \text{ is a face of } P\}
$
gives a complete rational fan in $N_\mathbb{R}$.
The fan $\Sigma(P)$ determines an $n$-dimensional toric variety $X_{\Sigma(P)}$,
which will be shortly denoted by  $X_P$.
In this paper,
we call $X_P$ the toric 
variety determined by  $P$.

If an $n$-dimensional convex polytope $P\subset M_\mathbb{R}$ contains the origin as an interior point, 
the polar dual $P^\circ \subset N_\mathbb{R}$ of $P$ is defined by $\{\nu\in N_\mathbb{R}\mid\langle \nu,m\rangle \geq -1 \text{ for all } m\in P \}$.
We can see that 
$P^\circ$ is also  an $n$-dimensional convex polytope and contains the origin as the interior point.
Moreover, the relation $(P^\circ)^\circ =P$ holds.
If a convex lattice polytope $P$ contains the origin $0\in M$ in its interior and $P^\circ$ gives a lattice polytope,
then $P$ is called a reflexive polytope. 
Batyrev \cite{B} proves that 
the $n$-dimensional toric variety $X_P$ is a Fano variety with Gorenstein singularities 
whose anti-canonical sections are given by
\begin{align}\label{LatticePtsGenerators}
H^0(X_P, \mathcal{O}_{X_P} (-K_{X_P}) ) =\bigoplus_{m\in P\cap M} \mathbb{C} \chi^m.
\end{align}
Then, a general member of the linear system $|-K_{X_P}|$ defines an $(n-1)$-dimensional Calabi-Yau hypersurface in $X_P$ as in \cite{B} Section 3.

Let $P\subset M_\mathbb{R}$ be a reflexive polytope.
For a face $\mathfrak{F}\subset P^\circ \subset N_\mathbb{R}$, let $\mathbb{R}_{\geq 0} \mathfrak{F}$ be the cone $\{\lambda x\mid x\in \mathfrak{F}, \lambda \in \mathbb{R}_{\geq 0}  \}$.   
Then, the set $\Sigma[P^\circ]$ of $\mathbb{R}_{\geq 0} \mathfrak{F}$ for all faces $\mathfrak{F}$ of $P^\circ$ gives a complete fan in $N_\mathbb{R}$ such that
$X_{\Sigma[P^\circ]} = X_{\Sigma(P)}=X_P$
(\cite{B} Proposition 2.1.1, Corollary 4.1.11).

We introduce Fano polytopes as follows.
This terminology is due to   \cite{BToric4}, \cite{Sa}  and \cite{Ob}.

\begin{df}\label{dfFano}
An $n$-dimensional convex  integral polytope $P\subset M_\mathbb{R}$ is called a Fano polytope,
if $P$ satisfies the following conditions:
\begin{itemize}

\item[(i)] $0 \in M_\mathbb{R}$ is a inner point of $P$,

\item[(ii)] the vertices of every facet of $P$ give a $\mathbb{Z}$-basis of $M$.

\end{itemize}
\end{df}

If $P$ is a Fano polytope,
$X_{P^\circ}=X_{\Sigma(P^\circ)}=X_{\Sigma[P]}$ is a smooth toric Fano $n$-fold.

\subsection{Dolgachev conjecture for toric $K3$ hypersurfaces}

In this paper, we will study cases of $n=3$.
In such cases, we will consider $K3$ hypersurfaces corresponding to the anti-canonical sections of three-dimensional toric varieties. 
We will study  lattice theoretic properties of $K3$ surfaces.
In this paper, let $U$ be the unimodular hyperbolic lattice of rank two.
 Also, let $E_8 (-1)$ be  the even unimodular lattice of signature $(0,8)$.
 Then, the $K3$ lattice $L_{K3}$, which is isomorphic to the $2$-homology group of a $K3$ surface, is isometric to the even unimodular lattice
 $II_{3,19}\simeq U^{\oplus 3}\oplus   E_8(-1)^{\oplus 2}$ of signature $(3,19)$. 
For a $K3$ surface $S$, its N\'eron-Severi lattice is denoted by ${\rm NS}(S)$ and its transcendental lattice is denoted by ${\rm Tr}(S).$

For a three-dimensional reflexive polytope $P$,
a very general member  of $|-K_{X_P}|$ defines a $K3$ surface $S_P$.
Also, we obtain another $K3$ surface $S_{P^\circ}$ from a very general member of $|-K_{X_{P^\circ}}|$.
Let $L_P$ ($L_{P^\circ}$, resp.) be the lattice in $L_{K3}$ which is isometric to the  sublattice  generated by the elements of $\iota^* ({\rm NS}(X_P))$ ($\iota_\circ ^* ({\rm NS}(X_{P^\circ}))$, resp.), 
where $\iota^* : H^2(X_P,\mathbb{Z}) \rightarrow H^2(S_P,\mathbb{Z})$ ($\iota_\circ^* : H^2(X_{P^\circ},\mathbb{Z}) \rightarrow H^2(S_{P^\circ},\mathbb{Z})$, resp.)
is the pull back of $\iota: S_P \hookrightarrow X_P$ ($\iota_\circ : S_{P^\circ} \hookrightarrow X_{P^\circ}$, resp.).
Let $L_P^\perp$ ($L_{P^\circ}^\perp $, resp.) be the orthogonal complement of $L_P$ ($L_{P^\circ}$, resp.) in $L_{K3}$.

The following conjecture is  famous and important to study mirror symmetry for toric $K3$ hypersurfaces.

\begin{conj}(The Dolgachev conjecture, \cite{D} Conjecture (8.6))
The notation being as above,
the lattice  $L_{P^\circ}^\perp$ has an orthogonal decomposition
$ L_{P^\circ}^\perp \simeq U \oplus \widecheck{L}_{P^\circ} $ with the summand $U$
and there is a primitive embedding $L_P \hookrightarrow \widecheck{L}_{P^\circ}$.
Moreover, $L_P = \widecheck{L}_{P^\circ}$ holds if and only if $L_{P^\circ} \simeq {\rm NS}(S_{P^\circ})$.
\end{conj}

We will study the cases for  three-dimensional Fano polytopes.
As in Table 1,
three-dimensional Fano polytopes are classified into 18 types  up to the action of ${\rm GL}_3(\mathbb{Z})$ by \cite{BToric} and \cite{WW}.
Suppose $P$ is a three-dimensional Fano polytope.
One can determine the structure of the N\'eron-Severi lattice ${\rm NS}(S_{P^\circ})$ of the $K3$ surface $S_{P^\circ}$
on the basis of the intersection theory for  divisors of toric manifolds.
We already have the following fact.

\begin{prop} (\cite{K2} or \cite{Mase})\label{PropLatticeDual}
 For $k\in \{1,\ldots, 18\},$
 let $P_k$ be the Fano polytope of Table 1. 
 Then,
the Neron-S\'everi lattice ${\rm NS}(S_{P_k^\circ})$ is isometric to the lattice $L_k$ of signature $(1,\ell_k-4)$ of Table 1,
where $L_k$ is equal to the lattice $\iota_\circ^* ({\rm NS} (X_{P_k^\circ}))$.
\end{prop}

In practice,   \cite{K2} proves the proposition by applying the following  result.

 \begin{prop} (\cite{Mo} Theorem 7.5)\label{PropMoishezon}
Let $X$ be a three-dimensional smooth  projective algebraic variety.
Let $\iota : S\hookrightarrow X$ be an embedding of a very general hyperplane section.
Then, the pull-back $\iota^*: {\rm NS}(X) \rightarrow {\rm NS}(S)$ is a surjective mapping 
if and only if 
$X $ and $S$ satisfy the following (i) or (ii):
\begin{itemize}

\item[(i)] $b_2(X)=b_2(S)$,

\item[(ii)] $h^{2,0} (S) > h^{2,0}(X)$.

\end{itemize} 
\end{prop}

We have $h^{2,0}(S_{P^\circ})=1$. 
Also, a smooth Fano three-fold $X_{P^\circ}$ satisfies $h^{2,0}(X_{P^\circ})=0$.
So,  one can apply Proposition \ref{PropMoishezon}
and calculate the intersection form of ${\rm NS}(S_{P^\circ})$  from that of ${\rm NS}(X_{P^\circ})$.

\begin{rem} \label{RemCheck}
By calculating the invariant of $L_k$ (see Section 4) for each $k$,
 we can see that the orthogonal complement $L_k^\perp$ admits a decomposition in the form $U\oplus \widecheck{L}_k$,
where $\widecheck{L}_k$ is a lattice of signature $(1,22-\ell_k)$.
 \end{rem}

 Furthermore, according to \cite{Kob} Corollary 4.3.6, 
  Proposition  \ref{PropLatticeDual} implies the isomorphism $L_{P_k}\simeq {\rm NS}(S_{P_k})$.
 Therefore, in order to see that the Dolgachev conjecture holds for three-dimensional Fano polytopes,
 it is enough to prove 
 \begin{align}\label{mirrorK3}
 {\rm Tr}(S_{P}) \simeq U \oplus {\rm NS}(S_{P^\circ}) =U \oplus L_{P^\circ}
 \end{align}
 for each  three-dimensional Fano polytope $P$.
If (\ref{mirrorK3}) holds, we have ${\rm Tr}(S_{P^\circ}) \simeq U \oplus {\rm NS}(S_{P})$.

A pair $(S,S')$ of  $K3$ surfaces is usually called a mirror pair if
$
{\rm Tr}(S) \simeq U \oplus {\rm NS}(S').
$
The relation (\ref{mirrorK3})  means that $(S_P, S_{P^\circ})$ is a mirror pair.

\subsection{$K3$  hypersurfaces with explicit parameters}

For each Fano polytope $P_k$ in Table 1,
let  $S_k=S_{P_k}$ be a very general member of the linear system $ |-K_{X_{P_k}}| $ in the toric three-fold $X_{P_k}$.
Due to Table 1,  Proposition \ref{PropLatticeDual} and (\ref{mirrorK3}),
it is enough to prove the relation
\begin{align}\label{K3mirrorF}
{\rm Tr}(S_k) \simeq U \oplus L_{k}
\end{align}
for $k\in \{1,\ldots, 18\}$ to see that the Dolgachev conjecture holds for the Fano polytopes.
As in Table 1,
the number of vertices of a three-dimensional Fano polytope is one of $4,5,6,7$ or $8$.

Let us recall  previous research.
The N\'eron-Severi lattice  of $S_1$, which is derived from the unique Fano polytope $P_1$ with $4$ vertices,
is determined by Narumiya and Shiga \cite{NS}. 
They show (\ref{K3mirrorF}) holds for $S_1$.
The motivation and the method of their paper are based on the study of the monodromy group of the Picard-Fuchs equation for the family of $S_1$.
The second author was inspired by \cite{NS}. 
Also, he was interested in a numerical approach of  Ishige,
 who studies the monodromy group of the Picard-Fuchs system for $K3$ surfaces derived from the Fano polytope  $P_3$ in detail.
Affected by these earlier studies,
 the second author \cite{Na}  determines the the N\'eron-Severi lattices coming from the Fano polytopes with $5$ vertices 
and studies the monodromy groups of the Picard-Fuchs systems.
His proof is based on techniques of Mordell-Weil lattices introduced in \cite{Shioda}.
Anyway,
in these works,
it is established  that the Dolgachev conjecture is true for the Fano polytopes with $4$ or $5$ vertices.

\begin{rem}
In  \cite{I} and \cite{IIT},
one can find the above mentioned Ishige's numerical approach to study the monodromy group.
We remark that there are several other computer-assisted studies of Picard-Fuchs systems coming from Fano polytopes.
For example,
Nakayama and Takayama \cite{NT} 
give an approximation algorithm to compute 
 Picard-Fuchs systems attached to  three-dimensional Fano polytopes 
via techniques of $D$-modules.
\end{rem}

In the present paper,
we will see that the Dolgachev conjecture is true for all the Fano polytopes with $6,7$ or $8$ vertices.
Namely, we will study toric $K3$ hypersurfaces $S_k$ $(k\in \{6,\ldots, 18\}).$
In this subsection, we will give explicit forms  of them.

Suppose a three-dimensional Fano polytope $P_k$ is given by
\begin{align}\label{polytopeP}
P_k=\begin{pmatrix} p_{11} & \cdots & p_{1\ell_k} \\ p_{21} & \cdots & p_{2\ell_k} \\ p_{31} & \cdots & p_{3\ell_k} \end{pmatrix},
\end{align}
 where $\ell_k$ is the number of vertices of $P_k$.
According to (\ref{LatticePtsGenerators}),
the family of the $K3$ hypersurfaces $S_k$ is explicitly given by
\begin{align}\label{EquationA}
\sum_{i=1}^{\ell_k} c_i t_1^{p_{1i}} t_2^{p_{2i}} t_3^{p_{3i}}=0, \quad \quad  c_1, \ldots, c_{\ell_k} \in \mathbb{C}.
\end{align}
Each Fano polytope in Table 1 contains the origin as the unique inner lattice point and the standard simplex generated by ${}^t (1,0,0), {}^t (0,1,0), {}^t (0,0,1)$.
From the polytope $P_k$ of (\ref{polytopeP}), together with Table 1,
we set
\begin{align}
\widetilde{P}_k=
\begin{pmatrix}
1&1&\cdots &1 \\
p_{10}&p_{11}& \cdots &p_{1\ell_k} \\
p_{20}&p_{21}& \cdots &p_{2\ell_k} \\
p_{30}&p_{31}& \cdots &p_{3\ell_k} 
\end{pmatrix}
\end{align}
 where ${\small \begin{pmatrix} p_{10} & p_{11} & p_{12} & p_{13} \\ p_{20} & p_{21} & p_{22} & p_{23} \\ p_{30} & p_{31} & p_{32} & p_{33}  \end{pmatrix} =\begin{pmatrix}  0&1&0&0 \\ 0&0&1&0 \\ 0&0&0&1 \end{pmatrix} }$.
 The matrix $\widetilde{P}_k$ is regarded as a surjective linear mapping $\mathbb{R}^{\ell_k+1} \rightarrow \mathbb{R} \oplus M_\mathbb{R} $
 which induces the exact sequence 
 \begin{align}\label{Sequence1}
0 \rightarrow \mathbb{Z}^{\ell_k-3} \overset{K_k} \longrightarrow \mathbb{Z}^{\ell_k+1} \overset{\widetilde{P}_k}\longrightarrow \mathbb{Z} \oplus M \rightarrow 0,
 \end{align}
where $K_k=(\kappa_{s,t})_{1 \leq s \leq \ell_k-3, 0 \leq t \leq \ell_k }$ is a $(\ell_k+1)\times (\ell_k-3)$-matrix with entries in $\mathbb{Z}$ 
such that its minor matrix $(\kappa_{s,t})_{1 \leq s \leq \ell_k-3, 4 \leq t \leq \ell_k }$ is equal to the identity matrix of rank $\ell_k - 3$.
Then, by putting 
$
\displaystyle
x=\frac{c_1 t_1}{c_0},  y=\frac{c_2 t_2}{c_0}$ and $\displaystyle z=\frac{c_3 t_3}{c_0},
$
the equation in (\ref{EquationA}) is transformed to
\begin{align}\label{EqStack}
S_k(\lambda_1,\ldots, \lambda_{\ell_k-3}): \quad 1 + x + y + z + \sum_{i=1}^{\ell_k-3} \lambda_i x^{p_{1,i+3}} y^{p_{2,i+3}} z^{p_{3,i+3}} =0.
\end{align}
Here, we put
\begin{align}\label{lambda}
\lambda_i = \prod_{h=0}^{\ell_k} c_h^{\kappa_{i+3, h}}.
\end{align}

By clearing denominators of (\ref{EqStack})
for the cases of $k\in \{6,\ldots,18\}$, 
we obtain explicit defining equations of $S_k(\lambda_1,\ldots,\lambda_{\ell_k -3})$ as in Table 2.

\vspace{5mm}
\begin{longtable}{ll}
\caption{Defining equations of toric $K3$ hypersurfaces $S_k$}
 \vspace{-1.5mm}
 \\
\hline
  $k$ & Defining equations of $S_k (\lambda_1,\ldots,\lambda_{\ell_k -3})$   
    \\
  \hline
  \endhead
\vspace{1mm}$6$ & $x y z (x + y + z+1) +\lambda_1 x y + \lambda_2 x z + \lambda_3 y z = 0$  \\
\vspace{1mm}$7$ & $xyz(x+y+z+1)+\lambda_1xy+\lambda_2x+\lambda_3y= 0$\\
\vspace{1mm}$8$ & $x y z (x+y+z+1)+\lambda_1 x y+\lambda_2 x z^2+\lambda_3 y=0$\\
\vspace{1mm}$9$ & $x y z (x+y+z+1)+\lambda_1 x y+\lambda_2 x z+\lambda_3 z=0$\\
\vspace{1mm}$10$ & $x y z (x+y+z+1)+\lambda_1 x y+\lambda_2 x+\lambda_3=0$\\
\vspace{1mm}$11$ & $x y z (x+y+z+1)+\lambda_1 x y+\lambda_2 x y^2+\lambda_3=0$\\
\vspace{1mm}$12$ & $x y z (x+y+z+1)+\lambda_1 x y+\lambda_2 x y^2+\lambda_3 z=0$\\
\vspace{1mm}$13$ & $x y z (x+y+z+1)+\lambda_1 x y+\lambda_2 x z+\lambda_3 y z+\lambda_4 z=0$\\
\vspace{1mm}$14$ & $x y z (x + y + z+1) +\lambda_1 y + \lambda_2 x z + \lambda_3 y z + \lambda_4 z=0$\\
\vspace{1mm}$15$ & $ x y z
 (x+y+z+1)+\lambda_1 x^2 y+\lambda_2 x z+\lambda_3 y z+\lambda_4 z=0$\\
 \vspace{1mm}$16$ & $x y z (x+y+z+1)+ \lambda_1+\lambda_2 x z+\lambda_3 y z+\lambda_4 z=0$\\
  \vspace{1mm}$17$ & $x y z (x+y+z+1)+\lambda_1 x y+\lambda_2 y z+\lambda_3 x z+\lambda_4 x^2 z+\lambda_5 y^2 z=0$\\
  $18$ & $x y z (x+y+z+1)+\lambda_1 x^2+\lambda_2 y z+\lambda_3 x z+\lambda_4 x^2
 z+\lambda_5 y^2 z=0$\\
\hline\\
\end{longtable}

We are able to give meaning to the above parameters $\lambda_1,\ldots, \lambda_{\ell_k-3}$ coming from the polytope $P_k$.
The matrix $K_k$ of (\ref{Sequence1}) is called the Gale transform of $\widetilde{P}_k$.
Each row of the matrix $K_k$ in (\ref{Sequence1}) is a $\mathbb{Z}$-vector of $\mathbb{R}^{\ell_k-3}$.
Let  $\Sigma'(P_k)$ be a fan in $\mathbb{R}^{\ell_k-3}$ 
whose one-dimensional cones are generated by these vectors.
This fan $\Sigma'(P_k)$ is regarded as  the secondary fan coming from  the set of  lattice points of the Fano polytope $P_k$
(for detail, see \cite{GKZ}; see also \cite{BFS} Section 4).
Then, the tuple of the parameters $\lambda_1,\ldots, \lambda_{\ell_k-3}$ of (\ref{lambda}) gives a system of coordinates of the torus $(\mathbb{C}^\times)^{\ell_k-3}$ of the $(\ell_k-3)$-dimensional toric variety determined by the secondary fan $\Sigma' (P_k)$.
Namely,  $(\mathbb{C}^\times)^{\ell_k-3}$ is regarded as ${\rm Spec}\left(\mathbb{C}\left[\lambda_1^{\pm},\ldots, \lambda_{\ell_k-3}^{\pm}\right]\right).$

Thus,
for each $k \in \{6,\ldots, 18\}$,
we obtain a family 
\begin{align}\label{Family}
\mathscr{F}_k: \{S_k (\lambda_1,\ldots, \lambda_{\ell_k-3}) \text{ of (\ref{EqStack})} \mid (\lambda_1,\ldots, \lambda_{\ell_k-3})\in (\mathbb{C}^\times)^{\ell_k-3} \} \rightarrow (\mathbb{C}^\times)^{\ell_k-3}
\end{align}
of toric $K3$ hypersurfaces.
We remark that such a family can  also be obtained via a morphism of fans precisely studied in \cite{Lafforgue}.

 \section{Jacobian elliptic fibrations and evident lattices}

 Let $S$ be a $K3$ surface.
 We call a triple $(S,\pi,C)$ an elliptic $K3$ surface, 
 if $C$ is a curve, $\pi:S\rightarrow C$ is a proper mapping such that the fibres  $\pi^{-1}(p)$ are elliptic curves for almost all $p\in C$  and $S$ is relatively minimal.
We suppose that an elliptic fibration $\pi:S\rightarrow C$ has a section $O:C\rightarrow S$ such that $O$ is the identity element of the Mordell-Weil group ${\rm MW}(\pi,O)$ of sections of $\pi$ (see  Section 5).
We call such an elliptic fibration $\pi$ a Jacobian elliptic fibration on $S$.
 
 We will study elliptic $K3$ surfaces $\pi:S\rightarrow \mathbb{P}^1(\mathbb{C}) $ induced from  cubic curves defined over the field $\mathbb{C}(x_1)$ of rational functions. 
 Here, $\mathbb{P}^1(\mathbb{C}) $ is regarded as the $x_1\text{-sphere}$.
 Suppose such an  elliptic curve is given by 
 $
 Z^2 W = Y^3 + a_1 Y^2 W + a_2 Y W^2 + a_3 W^3,
 $
 where $(W: Y: Z) $ is a system of homogeneous coordinates of $\mathbb{P}^2(\mathbb{C})$ and $a_1,a_2,a_3\in \mathbb{C}(x_1)$,
 then the point $(W:Y:Z)=(0:0:1)$ corresponds to the identity element $O$  of ${\rm MW}(\pi,O)$.
 From now on, we often use the affine form
 $
 z^2 = y^3 + a_1 y^2  + a_2 y  + a_3.
 $
 
In this section, we introduce an appropriate Jacobian elliptic fibration on the $K3$ surface $S_k (\lambda_1,\ldots, \lambda_{\ell_k-3})$ in Table 2 for each $k\in\{6,\ldots, 18\}$. 
From now on,
$S_k (\lambda_1,\ldots, \lambda_{\ell_k-3})$ will be shortly denoted by
$S_k(\lambda)$.

\begin{prop}\label{PropEquationJac}
For each $k\in \{6,\ldots, 18\},$
there is a Jacobian elliptic fibration 
$\pi_k=\pi_k^{(\lambda)}:S_k(\lambda) \rightarrow \mathbb{P}^1 (\mathbb{C})$
 defined by the equation
\begin{align}\label{EquationJac}
z_1^2 =4 y_1^3 + a_1 (x_1) y_1^2 + a_2(x_1) y_1 + a_3(x_1),
\end{align}
where $a_j (x_1)$ $(j=1,2,3)$ are polynomials given in Table 3.
\end{prop}

\vspace{-2.5mm}
{\small
\begin{longtable}{llll}
  \caption{\normalsize{$a_1(x_1), a_2(x_1)$ and $a_3(x_1)$ of (\ref{EquationJac})}}
  \vspace{-2.5mm}\\
\hline
  $k$ &  $a_1(x_1) $ &$a_2(x_1)$ & $a_3(x_1)$    
     \\
  \hline
  \endhead
\vspace{1mm}$6$ & $\lambda_2^2 + 2 \lambda_2 x_1 (1 + x_1) + x_1^2 (-4 \lambda_1 - 4 \lambda_3 + (1 + x_1)^2)$ & $4 \lambda_1 \lambda_3  x_1^4$ & $0$  \\
\vspace{1mm}$7$ & $x_1 (-4 \lambda_2 + x_1 (-4 \lambda_1 + (1 + x_1)^2))$ & $-2 \lambda_3 x_1^4 (1 + x_1)$ & $\lambda_3^2 x_1^6$ \\
\vspace{1mm}$8$ & $x_1 (x_1 (1 + x_1)^2 - 4 \lambda_1 (\lambda_2 + x_1))$ & $-2 \lambda_3 x_1^3 (1 + x_1) (\lambda_2 + x_1)$ & $\lambda_3^2 x_1^4 (\lambda_2 + x_1)^2$ \\
\vspace{0mm}$9$ & $\lambda_2^2 + 2 \lambda_2 x_1 (1 + x_1) $ & $4 \lambda_1 \lambda_3  x_1^3$ & $0$ \\
\vspace{1mm}& $+ x_1 (-4 \lambda_3 + x_1 (-4 \lambda_1 + (1 + x_1)^2))$ & & \\
\vspace{1mm}$10$ & $\lambda_1^2 + 2 \lambda_1 x_1 (1 + x_1) + x_1 (-4 \lambda_2 + x_1 (1 + x_1)^2)$ & $-2 \lambda_3 x_1^2 (\lambda_1 + x_1 + x_1^2)$ & $\lambda_3^2 x_1^4$ \\
\vspace{1mm}$11$ & $x_1^2 (1 - 4 \lambda_1 + (2 - 4 \lambda_2) x_1 + x_1^2)$ & $-2 \lambda_3 x_1^3 (1 + x_1)$ & $\lambda_3^2 x_1^4$\\
\vspace{1mm}$12$ & $(\lambda_1 + x_1 + x_1^2)^2$ & $-2 \lambda_3 x_1^2 (\lambda_2 + x_1) (\lambda_1 + x_1 + x_1^2)$ & $\lambda_3^2 x_1^4 (\lambda_2 + x_1)^2$\\
\vspace{0mm}$13$ & $\lambda_2^2 + 2 \lambda_2 x_1 (1 + x_1) $ & $4 \lambda_1 x_1^3 (\lambda_4 + \lambda_3 x_1)$ & $0$\\
\vspace{1mm} & $+ x_1 (-4 \lambda_4 + x_1 (-4 \lambda_1 - 4 \lambda_3 + (1 + x_1)^2))$ & &  \\
\vspace{0mm}$14$ & $\lambda_2^2 + 2 \lambda_2 x_1 (1 + x_1) $ & $-2 \lambda_1 x_1^3 (\lambda_2 + x_1 + x_1^2)$ & $\lambda_1^2 x_1^6$\\
\vspace{1mm}& $+ x_1 (-4 \lambda_4 + x_1 (-4 \lambda_3 + (1 + x_1)^2))$ && \\
\vspace{0mm}$15$ & $\lambda_2^2 + 2 \lambda_2 x_1 (1 + x_1) $ & $-2 \lambda_1 x_1^2 (\lambda_4 + \lambda_3 x_1) (\lambda_2 + x_1 + x_1^2)$ & $\lambda_1^2 x_1^4 (\lambda_4 + \lambda_3 x_1)^2$\\
\vspace{1mm}&$+ x_1 (-4 \lambda_4 + x_1 (-4 \lambda_3 + (1 + x_1)^2))$ && \\
 \vspace{0mm}$16$ & $\lambda_2^2 + 2 \lambda_2 x_1 (1 + x_1) $ & $-2 \lambda_1 x_1^2 (\lambda_2 + x_1 + x_1^2)$ & $\lambda_1^2 x_1^4$\\
\vspace{1mm} & $+ x_1 (-4 \lambda_4 + x_1 (-4 \lambda_3 + (1 + x_1)^2))$ && \\
  \vspace{0mm}$17$ & $\lambda_3^2 + 2 \lambda_3 x_1 (1 + x_1)  $ & $4 \lambda_1 x_1^3 (\lambda_4 + x_1) (\lambda_2 + \lambda_5 x_1)$ & $0$\\
\vspace{0mm}& $+x_1 (-4 \lambda_2 (\lambda_4 + x_1)  $ & & \\
\vspace{1mm} &$\hspace{7mm}+ x_1 (1 - 4 \lambda_1 - 4 \lambda_4 \lambda_5 + 2 x_1 - 4 \lambda_5 x_1 + x_1^2))$  & & \\
                        $18$ &  $\lambda_2^2 + 2 \lambda_2 x_1 (1 + x_1)  $ & $-2 \lambda_1 x_1^3 (\lambda_5 + x_1) (\lambda_2 + x_1 + x_1^2)$ & $\lambda_1^2 x_1^6 (\lambda_5 + x_1)^2$\\
 &$+ 
 x_1 (-4 \lambda_3 (\lambda_5 + x_1) $ && \\
 & $\hspace{7mm} + x_1 ((1 + x_1)^2 - 4 \lambda_4 (\lambda_5 + x_1)))$ & & \\
\hline\\
\end{longtable} 
}

\begin{proof}
By performing an appropriate birational transformation $(x,y,z)\mapsto (x_1,y_1,z_1) $ for each $k$ 
 as in Table A.1.1, A.1.2 and A.1.3 of Appendix, we have the assertion.
 See also \cite{M1}.
\end{proof}

 \begin{rem}
 There are several Jacobian elliptic fibrations on $S_k$ for each $k\in \{6,\ldots,18\}$.
 In fact, our $\pi_k$ is useful  to determine the lattice structure of $S_k$.
Namely, our mathematical arguments in Section 3, 4  and 5  heavily depend on the properties of $\pi_k$.
  \end{rem}

We have the following proposition by direct observations.

\begin{prop}
For each $k\in \{6,\ldots,18\},$
there are singular fibres of the elliptic fibration $\pi_k^{(\lambda)}$ on a  general member $S_k(\lambda)$ of the family $\mathscr{F}_k$ induced from the equation (\ref{EquationJac})
as listed in Table 4.
Also, there are non-trivial sections $\mathbb{P}^1 (\mathbb{C}) \rightarrow S_k$ of $\pi_k^{(\lambda)}$ given by $x_1 \mapsto (x_1,y_1,z_1)$  as in Table 4. 
\end{prop}

{\small
\begin{longtable}{lll}
  \caption{\normalsize{Singular fibres and non-trivial sections  of  $\pi_k$ }}
  \vspace{-5.5mm}\\
  \\
\hline
  $k$ & Kodaira type of singular fibres &  Non-trivial sections   $x_1 \mapsto (x_1,y_1,z_1)$   
    \\
  \hline
  \endhead
\vspace{1mm}$6$ &$I_8 + I_8 +8 I_1$ & $O': x_1\mapsto (x_1,0,0)$   \\
\vspace{1mm} & & $Q: x\mapsto (x_1,\lambda_1 x_1^2,\lambda_1 x_1^2 (\lambda_2 + x_1 +x_1^2))$\\
\vspace{1mm}$7$ & $I_3^*  +I_8 +7 I_1$ & $Q: x_1\mapsto (x_1,0,\lambda_3 x_1^3)$   \\
\vspace{1mm}$8$  & $I_1^* + I_3 + I_8 + 6 I_1$  & $Q: x_1\mapsto (x_1,0,\lambda_3 x_1^2 (\lambda_2+x_1) )$   \\
\vspace{1mm}$9$ & $I_6 + I_{10} +8 I_1$ & $O': x_1\mapsto (x_1,0,0)$   \\
\vspace{1mm}&&  $Q: x_1\mapsto (x_1,\lambda_1 x_1^2,\lambda_1 x_1^2 (\lambda_2+x_1 + x_1^2))$ \\
\vspace{1mm}$10$  & $I_5 +I_{11}+8 I_1$ & $Q: x_1\mapsto (x_1,0,\lambda_3  x_1^2)$   \\
\vspace{1mm}$11$ &$IV^* + I_9 + 7 I_1$ & $Q: x_1\mapsto (x_1,0,\lambda_3  x_1^2)$   \\
\vspace{1mm}$12$ &$I_6 + I_3 + I_9 + 6 I_1$ & $Q: x_1\mapsto (x_1,0,\lambda_3  x_1^2 (\lambda_2+x_1))$   \\
\vspace{1mm}$13$ &$I_6 + I_2 +I_8 +  8I_1 $ & $O': x_1\mapsto (x_1,0,0)$   \\
\vspace{1mm}& & $Q: x_1\mapsto (x_1,x_1(\lambda_4+ \lambda_3 x_1),x_1(\lambda_4+\lambda_3 x_1) (\lambda_2+x_1 + x_1^2))$ \\
\vspace{1mm}$14$ &$I_7 +I_8 + 9 I_1$ & $Q: x_1\mapsto (x_1,0,\lambda_1 x_1^3)$   \\
\vspace{1mm}$15$ &$I_5 +I_3 +I_8 +8I_1$ & $Q: x_1\mapsto (x_1,0,\lambda_1 x_1^2 (\lambda_4 + \lambda_3 x))$   \\
\vspace{1mm}$16$ &$I_5 + I_{10} + 9I_1$ & $Q: x_1\mapsto (x_1,0,\lambda_1 x_1^2 )$   \\
\vspace{1mm}$17$ &$I_6 + I_2 +I_2 + I_6 + 8 I_1$ & $O': x_1\mapsto (x_1,0,0)$   \\
\vspace{1mm}& & $Q: x_1\mapsto (x_1,x_1(\lambda_2+ \lambda_5 x_1) (\lambda_4+x_1),x_1 (\lambda_2+\lambda_5 x_1)(\lambda_4 +x_1)(\lambda_3+x_1 +x_1^2))$ \\
 \vspace{1mm}    $18$ & $I_7 +I_3 +I_5 + 9I_1$ & $Q: x_1\mapsto (x_1,0,\lambda_1 x_1^3 (\lambda_5 +x_1) )$   \\
\hline\\
\end{longtable} 
}

\begin{rem}
In fact, there are certain loci in $(\mathbb{C}^\times)^{\ell_k -3}$ on which we have degenerate elliptic $K3$ surfaces.
For instance, one can obtain a locus on  which two singular fibres of type $I_1$  collapse into a fibre of type $I_2$ for each $k$.
A  general member of the family $\mathscr{F}_k$  is avoiding such loci and  has the singular fibres  in Table 4.
\end{rem}

From now on,
for a $K3$ surface $S$,
let us regard ${\rm NS}(S)$ as a sublattice of the $2$-homology group $H_2(S,\mathbb{Z}).$ 
For $P\in {\rm MW}(\pi,O)$, let $(P) \in {\rm NS}(S)$ denote the corresponding divisor.  

Each elliptic fibration $\pi_k$ is determined by the equation (\ref{EquationJac}) and the coefficients $a_j(x_1)$ $(j\in \{1,2,3\})$ in Table 3.
If  a fibre $\pi_k^{-1} (x_1')$ at $x_1' \in \mathbb{P}^1 (\mathbb{C})$ is a singular fibre,
one can determine the Kodaira type of it.
The equation (\ref{EquationJac}) is transformed to an equation
$z_2^2 = y_2^3 +a_1'(x_2) y_2^2 + a_2'(x_2) y_2 + a_3'(x_2)$ 
with $x_2=x_1-x_1'$.
By putting $y_2 = y_2' - \frac{1}{3} a_1' (x_2)$,
we have the Weierstrass model
$z_2^2 = y_2'^3 + \alpha (x_2) y_2' +\beta(x_2)$.
Then, by referring to \cite{Kod}, we can determine the Kodaira type of $\pi_k^{-1} (x_1')$.
Each Kodaira type is obtained by a sequence of appropriate blow-ups.
We can check how each non-trivial section of $\pi_k$ intersects components of singular fibres
by observing the blow-ups and the explicit form of the section appeared in Table 4.

A general fibre
for each fibration $\pi_k$ is denoted by $F$.
For each $k\in \{6,\ldots, 18\}$,
the dual graph of 
the  sections and
the fibres  of $\pi_k$
is illustrated in Table 5. 
Here, each dot $\bullet$ ($\odot$, resp.)  stands for a rational nodal curve (a general fibre of the elliptic fibration, resp.).

\begin{longtable}{lll}
\caption{Dual graphs of the elliptic fibration $\pi_k$ coming from (\ref{EquationJac})}
\vspace{-5.5mm}\\
\\
\hline
  $k$ & Dual graphs & Divisors in singular fibres
    \\
  \hline
  \endhead
\vspace{1mm}\raisebox{8.6em}{$6$} & 
{\normalsize
\begin{tikzpicture}[scale=0.5]
 \draw (0,1)--(-1,2)--(-1,4)--(0,5)--(1,4)--(1,2)--cycle;
 \draw (3,1)--(8,1);
 \draw (5,4)--(7,3);
\draw (5,4)--(1,3);
 \draw (3,5)--(8,5);
 \draw (8,1)--(7,2)--(7,4)--(8,5)--(9,4)--(9,2)--cycle;
 \draw (3,1)--(0,1);
 \draw (3,5)--(0,5);
 \draw (3,1)--(3,3.35);
  \draw (3,5)--(3,3.65);
  \draw (5,4)--(3,2.5);
 \draw (0,1)node{$\bullet$};
 \draw (-1,2)node{$\bullet$};
 \draw (-1,3)node{$\bullet$};
 \draw (-1,4)node{$\bullet$};
 \draw (0,5)node{$\bullet$};
 \draw (1,4)node{$\bullet$};
 \draw (1,3)node{$\bullet$};
 \draw (1,2)node{$\bullet$};
 \draw (3,1)node{$\bullet$};
 \draw (3,5)node{$\bullet$};
 \draw (5,4)node{$\bullet$};
 \draw (8,1)node{$\bullet$};
 \draw (7,2)node{$\bullet$};
 \draw (7,3)node{$\bullet$};
 \draw (7,4)node{$\bullet$};
 \draw (8,5)node{$\bullet$};
 \draw (9,2)node{$\bullet$};
 \draw (9,3)node{$\bullet$};
 \draw (9,4)node{$\bullet$};
 \draw (3,2.5)node{$\odot$};
 \draw (1,2)node[left]{$a_1$};
 \draw (1,3)node[left]{$a_2$};
 \draw (1,4)node[left]{$a_3$};
  \draw (0,5)node[left]{$a_4$};
\draw (-1,4)node[left]{$a_5$};
 \draw (-1,3)node[left]{$a_6$};
  \draw (-1,2)node[left]{$a_7$};
 \draw (0,1)node[left]{$a_0$};
\draw (5,4)node[below]{$(Q)$};
 \draw (3,1)node[below]{$(O)$};
 \draw (3,5)node[above]{$(O')$};
 \draw (8,1)node[right]{$b_0$};
 \draw (7,2)node[right]{$b_1$};
 \draw (7,3)node[right]{$b_2$};
 \draw (7,4)node[right]{$b_3$};
 \draw (8,5)node[above]{$b_4$};
 \draw (9,4)node[right]{$b_5$};
 \draw (9,3)node[right]{$b_6$};
 \draw (9,2)node[right]{$b_7$};
  \draw (3,2.5)node[left]{$F$};
\end{tikzpicture}
}&
\raisebox{5.0em}{{\normalsize $\begin{matrix} \pi_6^{-1}(0)=a_0+a_1+\cdots +a_7,\\ \pi_6^{-1}(\infty)=b_0+b_1+\cdots +b_7. \end{matrix}$}}\\
\vspace{1mm}
\raisebox{8.8em}{$7$} & 
{\normalsize
\begin{tikzpicture}[scale=0.5]
 \draw (-1,0)--(0,1)--(0,4)--(-1,5);
 \draw (0,4)--(1,5);
 \draw (0,1)--(1,0);
 \draw (1,0)--(6,0);
 \draw (3,0)--(3,5);
 \draw (1,5)--(3,5);
 \draw (3,5)--(5,4);
 \draw (6,0)--(5,1)--(5,4)--(6,5)--(7,4)--(7,1)--cycle;
 \draw (-1,0)node{$\bullet$};
 \draw (0,1)node{$\bullet$};
 \draw (0,2)node{$\bullet$};
 \draw (0,3)node{$\bullet$};
 \draw (0,4)node{$\bullet$};
 \draw (-1,5)node{$\bullet$};
 \draw (1,5)node{$\bullet$};
 \draw (1,0)node{$\bullet$};
 \draw (3,0)node{$\bullet$};
 \draw (3,5)node{$\bullet$};
 \draw (6,0)node{$\bullet$};
 \draw (5,1)node{$\bullet$};
 \draw (5,2.5)node{$\bullet$};
 \draw (5,4)node{$\bullet$};
 \draw (6,5)node{$\bullet$};
 \draw (7,1)node{$\bullet$};
 \draw (7,2.5)node{$\bullet$};
 \draw (7,4)node{$\bullet$};
 \draw (3,2.5)node{$\odot$};
 \draw (-1,0)node[left]{$a_1$};
 \draw (0,1)node[left]{$a_2$};
 \draw (0,2)node[left]{$a_3$};
\draw (0,3)node[left]{$a_4$};
 \draw (0,4)node[left]{$a_5$};
 \draw (-1,5)node[left]{$a_6$};
 \draw (1,5)node[left]{$a_7$};
 \draw (1,0)node[left]{$a_0$};
 \draw (3,0)node[below]{$(O)$};
 \draw (3,5)node[above]{$(Q)$};
 \draw (6,0)node[right]{$b_0$};
 \draw (5,1)node[right]{$b_1$};
 \draw (5,2.5)node[right]{$b_2$};
 \draw (5,4)node[right]{$b_3$};
 \draw (6,5)node[right]{$b_4$};
 \draw (7,1)node[right]{$b_7$};
 \draw (7,2.5)node[right]{$b_6$};
 \draw (7,4)node[right]{$b_5$};
 \draw (3,2.5)node[left]{$F$};
\end{tikzpicture}
}&
\raisebox{5.0em}{{\normalsize $\begin{matrix} \pi_7^{-1}(0)=a_0+a_1+\cdots +a_7,\\ \pi_7^{-1}(\infty)=b_0+b_1+\cdots +b_7. \end{matrix}$}}\\
\vspace{1mm}\raisebox{8.8em}{$8$} & 
{\normalsize
\begin{tikzpicture}[scale=0.5]
 \draw (0,1)--(1,2)--(1,3)--(0,4);
 \draw (1,3)--(2,4);
 \draw (1,2)--(2,1);
 \draw (2,1)--(3,0)--(8,0);
 \draw (3,0)--(3,5);
 \draw (2,4)--(3,5);
 \draw (3,5)--(4.3,3);
 \draw (3,5)--(7,4);
 \draw (5,2)--(4.3,3)--(5.7,3)--cycle;
 \draw (3,0)--(5,2);
 \draw (8,0)--(7,1)--(7,4)--(8,5)--(9,4)--(9,1)--cycle;
 \draw (0,1)node{$\bullet$};
 \draw (2,1)node{$\bullet$};
 \draw (1,2)node{$\bullet$};
 \draw (1,3)node{$\bullet$};
 \draw (0,4)node{$\bullet$};
 \draw (2,4)node{$\bullet$};
  \draw (3,0)node{$\bullet$};
 \draw (3,5)node{$\bullet$};
 \draw (5,2)node{$\bullet$};
 \draw (4.3,3)node{$\bullet$};
 \draw (5.7,3)node{$\bullet$};
 \draw (8,0)node{$\bullet$};
 \draw (7,1)node{$\bullet$};
 \draw (7,2.5)node{$\bullet$};
 \draw (7,4)node{$\bullet$};
 \draw (8,5)node{$\bullet$};
 \draw (9,1)node{$\bullet$};
 \draw (9,2.5)node{$\bullet$};
 \draw (9,4)node{$\bullet$};
 \draw (3,2.5)node{$\odot$};
 \draw (0,1)node[left]{$a_1$};
 \draw (1,2)node[left]{$a_2$};
 \draw (1,3)node[left]{$a_3$};
\draw (2,4)node[left]{$a_5$};
 \draw (0,4)node[left]{$a_4$};
 \draw (2,1)node[left]{$a_0$};
 \draw (5,2)node[right]{$b_0$};
 \draw (4.5,3)node[above]{$b_1$};
 \draw (5.7,3)node[right]{$b_2$};
 \draw (3,0)node[below]{$(O)$};
 \draw (3,5)node[above]{$(Q)$};
 \draw (8,0)node[right]{$c_0$};
 \draw (7,1)node[right]{$c_1$};
 \draw (7,2.5)node[right]{$c_2$};
 \draw (7,4)node[right]{$c_3$};
 \draw (7,5)node[right]{$c_4$};
 \draw (9,1)node[right]{$c_7$};
 \draw (9,2.5)node[right]{$c_6$};
 \draw (9,4)node[right]{$c_5$};
 \draw (3,2.5)node[left]{$F$};
\end{tikzpicture}
}&
\raisebox{5.0em}{{\normalsize $\begin{matrix} \pi_8^{-1}(0)=a_0+a_1+\cdots +a_5,\\  \hspace{-5mm} \pi_8^{-1} (-\lambda_2)=b_0 +b_1 +b_2, \\ \pi_8^{-1}(\infty)=c_0+c_1+\cdots +c_7. \end{matrix}$}}\\
\vspace{1mm}
\raisebox{8.8em}{$9$} & 
{\normalsize
\begin{tikzpicture}[scale=0.5]
 \draw (0,1)--(-1,2)--(-1,3)--(0,4)--(1,3)--(1,2)--cycle;
 \draw (3,0)--(8,0);
 \draw (5,2)--(7,3);
\draw (5,2)--(1,2);
 \draw (3,5)--(8,5);
 \draw (8,0)--(7,1)--(7,4)--(8,5)--(9,4)--(9,1)--cycle;
 \draw (3,0)--(0,1);
 \draw (3,5)--(0,4);
 \draw (3,1.8)--(3,0);
  \draw (3,2.1)--(3,5);
 \draw (5,2)--(3,2.5);
 \draw (0,1)node{$\bullet$};
 \draw (-1,2)node{$\bullet$};
 \draw (-1,3)node{$\bullet$};
 \draw (0,4)node{$\bullet$};
 \draw (1,3)node{$\bullet$};
 \draw (1,2)node{$\bullet$};
 \draw (3,0)node{$\bullet$};
 \draw (3,5)node{$\bullet$};
 \draw (5,2)node{$\bullet$};
 \draw (8,0)node{$\bullet$};
 \draw (7,1)node{$\bullet$};
 \draw (7,2)node{$\bullet$};
 \draw (7,3)node{$\bullet$};
 \draw (7,4)node{$\bullet$};
 \draw (8,5)node{$\bullet$};
 \draw (9,1)node{$\bullet$};
 \draw (9,2)node{$\bullet$};
 \draw (9,3)node{$\bullet$};
 \draw (9,4)node{$\bullet$};
 \draw (3,2.5)node{$\odot$};
 \draw (1,2)node[left]{$a_1$};
 \draw (1,3)node[left]{$a_2$};
 \draw (0,4)node[left]{$a_3$};
 \draw (-1,3)node[left]{$a_4$};
  \draw (-1,2)node[left]{$a_5$};
 \draw (0,1)node[left]{$a_0$};
\draw (5,2)node[below]{$(Q)$};
 \draw (3,0)node[below]{$(O)$};
 \draw (3,5)node[above]{$(O')$};
 \draw (8,0)node[right]{$b_0$};
  \draw (7,1)node[right]{$b_1$};
 \draw (7,2)node[right]{$b_2$};
 \draw (7,3)node[right]{$b_3$};
 \draw (7,4)node[right]{$b_4$};
 \draw (8,5)node[above]{$b_5$};
 \draw (9,4)node[right]{$b_6$};
 \draw (9,3)node[right]{$b_7$};
 \draw (9,2)node[right]{$b_8$};
  \draw (9,1)node[right]{$b_9$};
  \draw (3,2.88)node[right]{$F$};
\end{tikzpicture}
}&
\raisebox{5.0em}{{\normalsize $\begin{matrix} \pi_9^{-1}(0)=a_0+a_1+\cdots +a_5,\\ \pi_9^{-1}(\infty)=b_0+b_1+\cdots +b_9. \end{matrix}$}}\\
\vspace{1mm}
\raisebox{8.8em}{$10$} & 
 {\normalsize
\begin{tikzpicture}[scale=0.5]
 \draw (0,1)--(1,2)--(1,3)--(-1,3)--(-1,2)--cycle;
 \draw (0,1)--(3,0)--(6,0);
 \draw (3,0)--(3,5);
 \draw (1,3)--(3,5);
 \draw (3,5)--(5,4);
 \draw (6,0)--(5,1)--(5,5)--(7,5)--(7,1)--cycle;
 \draw (0,1)node{$\bullet$};
 \draw (1,2)node{$\bullet$};
 \draw (1,3)node{$\bullet$};
 \draw (-1,3)node{$\bullet$};
 \draw (-1,2)node{$\bullet$};
 \draw (3,0)node{$\bullet$};
 \draw (3,5)node{$\bullet$};
 \draw (6,0)node{$\bullet$};
 \draw (5,1)node{$\bullet$};
 \draw (5,2)node{$\bullet$};
 \draw (5,3)node{$\bullet$};
 \draw (5,4)node{$\bullet$};
 \draw (5,5)node{$\bullet$};
 \draw (7,1)node{$\bullet$};
 \draw (7,2)node{$\bullet$};
\draw (7,3)node{$\bullet$};
 \draw (7,4)node{$\bullet$};
 \draw (7,5)node{$\bullet$};
 \draw (3,2.5)node{$\odot$};
 \draw (1,2)node[left]{$a_1$};
 \draw (0.7,3)node[above]{$a_2$};
\draw (-1,3)node[left]{$a_3$};
 \draw (-1,2)node[left]{$a_4$};
 \draw (0,1)node[left]{$a_0$};
 \draw (3,0)node[below]{$(O)$};
 \draw (3,5)node[above]{$(Q)$};
 \draw (6,0)node[right]{$b_0$};
 \draw (5,1)node[right]{$b_1$};
 \draw (5,2)node[right]{$b_2$};
\draw (5,3)node[right]{$b_3$};
 \draw (5,4)node[right]{$b_4$};
 \draw (5,5)node[above]{$b_5$};
 \draw (7,1)node[right]{$b_{10}$};
 \draw (7,2)node[right]{$b_9$};
 \draw (7,3)node[right]{$b_8$};
 \draw (7,4)node[right]{$b_7$};
 \draw (7,5)node[right]{$b_6$};
 \draw (3,2.5)node[left]{$F$};
\end{tikzpicture}
}&
\raisebox{5.0em}{{\normalsize $\begin{matrix} \pi_{10}^{-1}(0)=a_0+a_1+\cdots +a_4,\\ \hspace{2mm} \pi_{10}^{-1}(\infty)=b_0+b_1+\cdots +b_{10}. \end{matrix}$}}\\
\vspace{1mm}\raisebox{8.8em}{$11$} & 
{\normalsize
\begin{tikzpicture}[scale=0.5]
 \draw (0,0)--(0,4);
 \draw (0,2)--(-2,2);
 \draw (3,0)--(3,5);
 \draw (3,5)--(5,4);
 \draw (0,0)--(3,0);
 \draw (0,4)--(3,5);
 \draw (3,0)--(6,0);
 \draw (6,0)--(5,1)--(5,4)--(7,4)--(7,1)--cycle;
 \draw (0,0)node{$\bullet$};
 \draw (0,1)node{$\bullet$};
 \draw (0,2)node{$\bullet$};
 \draw (0,3)node{$\bullet$};
 \draw (0,4)node{$\bullet$};
 \draw (-1,2)node{$\bullet$};
\draw (-2,2)node{$\bullet$};
 \draw (3,0)node{$\bullet$};
 \draw (3,5)node{$\bullet$};
 \draw (6,0)node{$\bullet$};
 \draw (5,1)node{$\bullet$};
 \draw (5,2)node{$\bullet$};
 \draw (5,3)node{$\bullet$};
 \draw (5,4)node{$\bullet$};
 \draw (7,1)node{$\bullet$};
 \draw (7,2)node{$\bullet$};
\draw (7,3)node{$\bullet$};
 \draw (7,4)node{$\bullet$};
 \draw (3,2.5)node{$\odot$};
 \draw (0,2)node[right]{$a_2$};
 \draw (0,1)node[left]{$a_1$};
\draw (0,3)node[left]{$a_3$};
 \draw (0,4)node[left]{$a_4$};
 \draw (0,0)node[left]{$a_0$};
 \draw (-1,2)node[above]{$a_5$};
 \draw (-2,2)node[above]{$a_6$};
 \draw (3,0)node[below]{$(O)$};
 \draw (3,5)node[above]{$(Q)$};
 \draw (6,0)node[right]{$b_0$};
 \draw (5,1)node[right]{$b_1$};
 \draw (5,2)node[right]{$b_2$};
\draw (5,3)node[right]{$b_3$};
 \draw (5,4.4)node[right]{$b_4$};
 \draw (7,1)node[right]{$b_{8}$};
 \draw (7,2)node[right]{$b_7$};
 \draw (7,3)node[right]{$b_6$};
 \draw (7,4)node[right]{$b_5$};
 \draw (3,2.5)node[left]{$F$};
\end{tikzpicture}
}&
\raisebox{5.0em}{{\normalsize $\begin{matrix} \pi_{11}^{-1}(0)=a_0+a_1+\cdots +a_6,\\ \pi_{11}^{-1}(\infty)=b_0+b_1+\cdots +b_8. \end{matrix}$}}\\
\vspace{1mm}\raisebox{8.8em}{$12$} & 
{\normalsize
\begin{tikzpicture}[scale=0.5]
 \draw (0,1)--(-1,2)--(-1,3)--(0,4)--(1,3)--(1,2)--cycle;
 \draw (3,0)--(8,0);
 \draw (3,0)--(3,5);
 \draw (3,5)--(4.3,3);
 \draw (3,5)--(7,3);
 \draw (5,2)--(4.3,3)--(5.7,3)--cycle;
 \draw (3,0)--(5,2);
 \draw (8,0)--(7,1)--(7,4)--(9,4)--(9,1)--cycle;
 \draw (3,0)--(0,1);
 \draw (3,5)--(1,3);
 \draw (0,1)node{$\bullet$};
 \draw (-1,2)node{$\bullet$};
 \draw (-1,3)node{$\bullet$};
 \draw (0,4)node{$\bullet$};
 \draw (1,3)node{$\bullet$};
 \draw (1,2)node{$\bullet$};
 \draw (3,0)node{$\bullet$};
 \draw (3,5)node{$\bullet$};
 \draw (5,2)node{$\bullet$};
 \draw (4.3,3)node{$\bullet$};
 \draw (5.7,3)node{$\bullet$};
 \draw (8,0)node{$\bullet$};
 \draw (7,1)node{$\bullet$};
 \draw (7,2)node{$\bullet$};
 \draw (7,3)node{$\bullet$};
 \draw (7,4)node{$\bullet$};
 \draw (9,1)node{$\bullet$};
 \draw (9,2)node{$\bullet$};
 \draw (9,3)node{$\bullet$};
 \draw (9,4)node{$\bullet$};
 \draw (3,2.5)node{$\odot$};
 \draw (1,2)node[left]{$a_1$};
 \draw (1,3)node[left]{$a_2$};
 \draw (0,4)node[left]{$a_3$};
\draw (-1,3)node[left]{$a_4$};
 \draw (-1,2)node[left]{$a_5$};
 \draw (0,1)node[left]{$a_0$};
 \draw (5,2)node[right]{$b_0$};
 \draw (4.5,3)node[above]{$b_1$};
 \draw (5.7,3)node[right]{$b_2$};
 \draw (3,0)node[below]{$(O)$};
 \draw (3,5)node[above]{$(Q)$};
 \draw (8,0)node[right]{$c_0$};
 \draw (7,1)node[right]{$c_1$};
 \draw (7,2)node[right]{$c_2$};
 \draw (7,3)node[right]{$c_3$};
 \draw (7,4)node[above]{$c_4$};
 \draw (9,4)node[right]{$c_5$};
 \draw (9,3)node[right]{$c_6$};
 \draw (9,2)node[right]{$c_7$};
 \draw (9,1)node[right]{$c_8$};
 \draw (3,2.5)node[left]{$F$};
\end{tikzpicture}
}&
\raisebox{5.0em}{{\normalsize $\begin{matrix} \pi_{12}^{-1}(0)=a_0+a_1+\cdots +a_5,\\  \hspace{-5mm} \pi_{12}^{-1} (-\lambda_2)=b_0 +b_1 +b_2, \\ \pi_{12}^{-1}(\infty)=c_0+c_1+\cdots +c_8. \end{matrix}$}}
\\
\vspace{1mm}\raisebox{8.8em}{$13$} & 
{\normalsize
\begin{tikzpicture}[scale=0.5]
 \draw (0,1)--(-1,2)--(-1,3)--(0,4)--(1,3)--(1,2)--cycle;
 \draw (3,0)--(9,0);
 \draw (9,0)--(8,1)--(8,3)--(9,4)--(10,3)--(10,1)--cycle;
 \draw (3,0)--(0,1);
 \draw (3,5)--(0,4);
 \draw (3,2.9)--(3,0);
  \draw (3,3.1)--(3,5);
 \draw (5,4)--(1,2);
 \draw (3,0)--(6,1);
 \draw (5,4)--(6,2);
 \draw (5.9,1)--(5.9,2);
 \draw (6.1,1)--(6.1,2);
 \draw (4.6,3.4)--(6,2);
  \draw (3,5)--(4.2,3.8);
 \draw (3,5)--(9,4);
 \draw (3,2.5)--(5,4);
 \draw (5,4)--(8,2);
 \draw (0,1)node{$\bullet$};
 \draw (-1,2)node{$\bullet$};
 \draw (-1,3)node{$\bullet$};
 \draw (0,4)node{$\bullet$};
 \draw (1,3)node{$\bullet$};
 \draw (1,2)node{$\bullet$};
 \draw (3,0)node{$\bullet$};
 \draw (3,5)node{$\bullet$};
 \draw (5,4)node{$\bullet$};
 \draw (6,1)node{$\bullet$};
\draw (6,2)node{$\bullet$};
 \draw (9,0)node{$\bullet$};
 \draw (8,1)node{$\bullet$};
 \draw (8,2)node{$\bullet$};
 \draw (8,3)node{$\bullet$};
 \draw (9,4)node{$\bullet$};
 \draw (10,1)node{$\bullet$};
 \draw (10,2)node{$\bullet$};
 \draw (10,3)node{$\bullet$};
 \draw (3,2.5)node{$\odot$};
 \draw (1,2)node[left]{$a_1$};
 \draw (1,3)node[left]{$a_2$};
 \draw (0,4)node[left]{$a_3$};
 \draw (-1,3)node[left]{$a_4$};
  \draw (-1,2)node[left]{$a_5$};
 \draw (0,1)node[left]{$a_0$};
\draw (5,4)node[right]{$(Q)$};
 \draw (3,0)node[below]{$(O)$};
 \draw (3,5)node[above]{$(O')$};
 \draw (6,1)node[right]{$b_0$};
 \draw (6,2)node[right]{$b_1$};
 \draw (9,0)node[right]{$c_0$};
  \draw (8,1)node[right]{$c_1$};
 \draw (8,2)node[right]{$c_2$};
 \draw (8,3)node[right]{$c_3$};
 \draw (9,4)node[right]{$c_4$};
 \draw (10,3)node[right]{$c_5$};
 \draw (10,2)node[right]{$c_6$};
  \draw (10,1)node[right]{$c_7$};
  \draw (3,2.3)node[right]{$F$};
\end{tikzpicture}
}&
\raisebox{5.0em}{{\normalsize $\begin{matrix} \pi_{13}^{-1}(0)=a_0+a_1+\cdots +a_5,\\  \hspace{-12.5mm} \pi_{13}^{-1} (-\frac{\lambda_4}{\lambda_3})=b_0 +b_1, \\ \pi_{13}^{-1}(\infty)=c_0+c_1+\cdots +c_7. \end{matrix}$ }}\\
\vspace{1mm}
\raisebox{8.8em}{$14$}  &
 {\normalsize
\begin{tikzpicture}[scale=0.5]
 \draw (0,1)--(1,2)--(1,4)--(-1,4)--(-1,2)--cycle;
 \draw (0,1)--(3,0)--(6,0);
 \draw (3,0)--(3,5);
 \draw (1,4)--(3,5);
 \draw (3,5)--(5,4);
 \draw (6,0)--(5,1)--(5,4)--(6,5)--(7,4)--(7,1)--cycle;
 \draw (0,1)node{$\bullet$};
 \draw (1,2)node{$\bullet$};
 \draw (1,3)node{$\bullet$};
 \draw (1,4)node{$\bullet$};
 \draw (-1,4)node{$\bullet$};
 \draw (-1,3)node{$\bullet$};
 \draw (-1,2)node{$\bullet$};
 \draw (3,0)node{$\bullet$};
 \draw (3,5)node{$\bullet$};
 \draw (6,0)node{$\bullet$};
 \draw (5,1)node{$\bullet$};
 \draw (5,2.5)node{$\bullet$};
 \draw (5,4)node{$\bullet$};
 \draw (6,5)node{$\bullet$};
 \draw (7,1)node{$\bullet$};
 \draw (7,2.5)node{$\bullet$};
 \draw (7,4)node{$\bullet$};
 \draw (3,2.5)node{$\odot$};
 \draw (1,2)node[left]{$a_1$};
 \draw (1,3)node[left]{$a_2$};
 \draw (0.8,4)node[above]{$a_3$};
\draw (-1,4)node[left]{$a_4$};
 \draw (-1,3)node[left]{$a_5$};
 \draw (-1,2)node[left]{$a_6$};
 \draw (0,1)node[left]{$a_0$};
 \draw (3,0)node[below]{$(O)$};
 \draw (3,5)node[above]{$(Q)$};
 \draw (6,0)node[right]{$b_0$};
 \draw (5,1)node[right]{$b_1$};
 \draw (5,2.5)node[right]{$b_2$};
 \draw (5,4)node[right]{$b_3$};
 \draw (6,5)node[right]{$b_4$};
 \draw (7,1)node[right]{$b_7$};
 \draw (7,2.5)node[right]{$b_6$};
 \draw (7,4)node[right]{$b_5$};
 \draw (3,2.5)node[left]{$F$};
\end{tikzpicture}
}&
\raisebox{5.0em}{{\normalsize $\begin{matrix} \pi_{14}^{-1}(0)=a_0+a_1+\cdots +a_6,\\ \pi_{14}^{-1}(\infty)=b_0+b_1+\cdots +b_7. \end{matrix}$}}
\\
\vspace{1mm}\raisebox{8.8em}{$15$} & 
{\normalsize
\begin{tikzpicture}[scale=0.5]
 \draw (0,1)--(1,2)--(1,3)--(-1,3)--(-1,2)--cycle;
 \draw (0,1)--(3,0);
 \draw (3,0)--(8,0);
 \draw (3,0)--(3,5);
 \draw (3,5)--(4.3,3);
 \draw (3,5)--(7,4);
 \draw (5,2)--(4.3,3)--(5.7,3)--cycle;
 \draw (3,0)--(5,2);
 \draw (1,3)--(3,5);
 \draw (8,0)--(7,1)--(7,4)--(8,5)--(9,4)--(9,1)--cycle;
 \draw (0,1)node{$\bullet$};
 \draw (1,2)node{$\bullet$};
 \draw (1,3)node{$\bullet$};
 \draw (-1,3)node{$\bullet$};
 \draw (-1,2)node{$\bullet$};
 \draw (3,0)node{$\bullet$};
 \draw (3,5)node{$\bullet$};
 \draw (5,2)node{$\bullet$};
 \draw (4.3,3)node{$\bullet$};
 \draw (5.7,3)node{$\bullet$};
 \draw (8,0)node{$\bullet$};
 \draw (7,1)node{$\bullet$};
 \draw (7,2.5)node{$\bullet$};
 \draw (7,4)node{$\bullet$};
 \draw (8,5)node{$\bullet$};
 \draw (9,1)node{$\bullet$};
 \draw (9,2.5)node{$\bullet$};
 \draw (9,4)node{$\bullet$};
 \draw (3,2.5)node{$\odot$};
 \draw (1,2)node[left]{$a_1$};
 \draw (0.9,3)node[above]{$a_2$};
 \draw (-1,3)node[left]{$a_3$};
\draw (-1,2)node[left]{$a_4$};
 \draw (0,1)node[left]{$a_0$};
 \draw (5,2)node[right]{$b_0$};
 \draw (4.5,3)node[above]{$b_1$};
 \draw (5.7,3)node[right]{$b_2$};
 \draw (3,0)node[below]{$(O)$};
 \draw (3,5)node[above]{$(Q)$};
 \draw (8,0)node[right]{$c_0$};
 \draw (7,1)node[right]{$c_1$};
 \draw (7,2.5)node[right]{$c_2$};
 \draw (7,4)node[right]{$c_3$};
 \draw (7,5)node[right]{$c_4$};
 \draw (9,1)node[right]{$c_7$};
 \draw (9,2.5)node[right]{$c_6$};
 \draw (9,4)node[right]{$c_5$};
 \draw (3,2.5)node[left]{$F$};
\end{tikzpicture}
}&
\raisebox{5.0em}{{\normalsize $\begin{matrix} \pi_{15}^{-1}(0)=a_0+a_1+\cdots +a_4,\\  \hspace{-5mm} \pi_{15}^{-1} (-\lambda_2)=b_0 +b_1 +b_2, \\ \pi_{15}^{-1}(\infty)=c_0+c_1+\cdots +c_7. \end{matrix}$}}\\
 \vspace{1mm}\raisebox{8.8em}{$16$} & 
  {\normalsize
\begin{tikzpicture}[scale=0.5]
 \draw (0,1)--(1,2)--(1,3)--(-1,3)--(-1,2)--cycle;
 \draw (0,1)--(3,0)--(6,0);
 \draw (3,0)--(3,5);
 \draw (1,3)--(3,5);
 \draw (3,5)--(5,4);
 \draw (6,0)--(5,1)--(5,4)--(6,5)--(7,4)--(7,1)--cycle;
 \draw (0,1)node{$\bullet$};
 \draw (1,2)node{$\bullet$};
 \draw (1,3)node{$\bullet$};
 \draw (-1,3)node{$\bullet$};
 \draw (-1,2)node{$\bullet$};
 \draw (3,0)node{$\bullet$};
 \draw (3,5)node{$\bullet$};
 \draw (6,0)node{$\bullet$};
 \draw (5,1)node{$\bullet$};
 \draw (5,2)node{$\bullet$};
 \draw (5,3)node{$\bullet$};
 \draw (5,4)node{$\bullet$};
 \draw (6,5)node{$\bullet$};
 \draw (7,1)node{$\bullet$};
 \draw (7,2)node{$\bullet$};
\draw (7,3)node{$\bullet$};
 \draw (7,4)node{$\bullet$};
 \draw (3,2.5)node{$\odot$};
 \draw (1,2)node[left]{$a_1$};
 \draw (0.7,3)node[above]{$a_2$};
\draw (-1,3)node[left]{$a_3$};
 \draw (-1,2)node[left]{$a_4$};
 \draw (0,1)node[left]{$a_0$};
 \draw (3,0)node[below]{$(O)$};
 \draw (3,5)node[above]{$(Q)$};
 \draw (6,0)node[right]{$b_0$};
 \draw (5,1)node[right]{$b_1$};
 \draw (5,2)node[right]{$b_2$};
\draw (5,3)node[right]{$b_3$};
 \draw (5,4)node[right]{$b_4$};
 \draw (6,5)node[right]{$b_5$};
 \draw (7,1)node[right]{$b_9$};
 \draw (7,2)node[right]{$b_8$};
 \draw (7,3)node[right]{$b_7$};
 \draw (7,4)node[right]{$b_6$};
 \draw (3,2.5)node[left]{$F$};
\end{tikzpicture}
}&
\raisebox{5.0em}{{\normalsize $\begin{matrix} \pi_{16}^{-1}(0)=a_0+a_1+\cdots +a_4,\\ \pi_{16}^{-1}(\infty)=b_0+b_1+\cdots +b_9. \end{matrix}$}}
\\
  \vspace{1mm}\raisebox{8.8em}{$17$} &
  {\normalsize
\begin{tikzpicture}[scale=0.5]
 \draw (0,1)--(-1,2)--(-1,3)--(0,4)--(1,3)--(1,2)--cycle;
 \draw (3,0)--(9,1);
 \draw (9,1)--(8,2)--(8,3)--(9,4)--(10,3)--(10,2)--cycle;
 \draw (3,0)--(0,1);
 \draw (3,5)--(0,4);
 \draw (3,2.9)--(3,0);
  \draw (3,3.1)--(3,5);
 \draw (5,4)--(1,2);
 \draw (3,0)--(6,1);
 \draw (5,4)--(6,2);
 \draw (5.9,1)--(5.9,2);
 \draw (6.1,1)--(6.1,2);
 \draw (3,0)--(4.5,1);
 \draw (4.4,1)--(4.4,2);
 \draw (4.6,1)--(4.6,2);
 \draw (4.55,3.45)--(4.7,3.3);
 \draw (4.9,3.1)--(6,2);
  \draw (3,5)--(4.2,3.8);
 \draw (3,5)--(9,4);
 \draw (3,2.5)--(5,4);
 \draw (5,4)--(8,2);
 \draw (3,5)--(3.7,3.6);
 \draw (4,3)--(4.5,2);
  \draw (3.84,3.32)--(3.88,3.24);
 \draw (5,4)--(4.5,2);
 \draw (0,1)node{$\bullet$};
 \draw (-1,2)node{$\bullet$};
 \draw (-1,3)node{$\bullet$};
 \draw (0,4)node{$\bullet$};
 \draw (1,3)node{$\bullet$};
 \draw (1,2)node{$\bullet$};
 \draw (3,0)node{$\bullet$};
 \draw (3,5)node{$\bullet$};
 \draw (5,4)node{$\bullet$};
 \draw (4.5,1)node{$\bullet$};
\draw (4.5,2)node{$\bullet$};
 \draw (6,1)node{$\bullet$};
\draw (6,2)node{$\bullet$};
 \draw (9,1)node{$\bullet$};
 \draw (8,2)node{$\bullet$};
 \draw (8,3)node{$\bullet$};
 \draw (9,4)node{$\bullet$};
 \draw (10,2)node{$\bullet$};
 \draw (10,3)node{$\bullet$};
 \draw (3,2.5)node{$\odot$};
 \draw (1,2)node[left]{$a_1$};
 \draw (1,3)node[left]{$a_2$};
 \draw (0,4)node[left]{$a_3$};
 \draw (-1,3)node[left]{$a_4$};
  \draw (-1,2)node[left]{$a_5$};
 \draw (0,1)node[left]{$a_0$};
\draw (5,4)node[right]{$(Q)$};
 \draw (3,0)node[below]{$(O)$};
 \draw (3,5)node[above]{$(O')$};
 \draw (4.5,1)node[right]{$b_0$};
 \draw (4.5,2)node[right]{$b_1$};
 \draw (6,1)node[right]{$c_0$};
 \draw (6,2)node[right]{$c_1$};
  \draw (9,1)node[right]{$d_0$};
 \draw (8,2)node[right]{$d_1$};
 \draw (8,3)node[right]{$d_2$};
 \draw (9,4)node[right]{$d_3$};
 \draw (10,3)node[right]{$d_4$};
 \draw (10,2)node[right]{$d_5$};
  \draw (3,2.1)node[left]{$F$};
\end{tikzpicture}
}&
\raisebox{5.0em}{{\normalsize $\begin{matrix} \pi_{17}^{-1}(0)=a_0+a_1+\cdots +a_5,\\  \hspace{-12.5mm} \pi_{17}^{-1} (-\lambda_4)=b_0 +b_1, \\ 
 \hspace{-11mm} \pi_{17}^{-1} (-\frac{\lambda_2}{\lambda_5})=c_0 +c_1,,\\ \pi_{17}^{-1}(\infty)=d_0+d_1+\cdots +c_5. \end{matrix}$}}\\
   \raisebox{8.8em}{$18$} & 
  {\normalsize
\begin{tikzpicture}[scale=0.5]
 \draw (0,1)--(-1,2)--(-1,4)--(1,4)--(1,2)--cycle;
 \draw (3,0)--(8,1);
 \draw (3,0)--(3,5);
 \draw (3,5)--(4.3,3);
 \draw (3,5)--(7,3);
 \draw (5,2)--(4.3,3)--(5.7,3)--cycle;
 \draw (3,0)--(5,2);
 \draw (8,1)--(7,2)--(7,3)--(9,3)--(9,2)--cycle;
 \draw (3,0)--(0,1);
 \draw (3,5)--(1,4);
 \draw (0,1)node{$\bullet$};
 \draw (-1,2)node{$\bullet$};
 \draw (-1,3)node{$\bullet$};
 \draw (-1,4)node{$\bullet$};
 \draw (1,4)node{$\bullet$};
 \draw (1,3)node{$\bullet$};
 \draw (1,2)node{$\bullet$};
 \draw (3,0)node{$\bullet$};
 \draw (3,5)node{$\bullet$};
 \draw (5,2)node{$\bullet$};
 \draw (4.3,3)node{$\bullet$};
 \draw (5.7,3)node{$\bullet$};
 \draw (8,1)node{$\bullet$};
 \draw (7,2)node{$\bullet$};
 \draw (7,3)node{$\bullet$};
 \draw (9,2)node{$\bullet$};
 \draw (9,3)node{$\bullet$};
 \draw (3,2.5)node{$\odot$};
 \draw (1,2)node[left]{$a_1$};
 \draw (1,3)node[left]{$a_2$};
 \draw (1,4)node[above]{$a_3$};
\draw (-1,4)node[left]{$a_4$};
 \draw (-1,3)node[left]{$a_5$};
 \draw (-1,2)node[left]{$a_6$};
 \draw (0,1)node[left]{$a_0$};
 \draw (5,2)node[right]{$b_0$};
 \draw (4.5,3)node[above]{$b_1$};
 \draw (5.7,3)node[right]{$b_2$};
 \draw (3,0)node[below]{$(O)$};
 \draw (3,5)node[above]{$(Q)$};
 \draw (8,1)node[right]{$c_0$};
 \draw (7,2)node[right]{$c_1$};
 \draw (7,3)node[above]{$c_2$};
 \draw (9,3)node[right]{$c_3$};
 \draw (9,2)node[right]{$c_4$};
 \draw (3,2.5)node[left]{$F$};
\end{tikzpicture}
}&
\raisebox{5.0em}{{\normalsize $\begin{matrix} \pi_{18}^{-1}(0)=a_0+a_1+\cdots +a_6,\\  \hspace{-5mm} \pi_{18}^{-1} (-\lambda_5)=b_0 +b_1 +b_2, \\ \pi_{18}^{-1}(\infty)=c_0+c_1+\cdots +c_4. \end{matrix}$}}\\
\hline\\
\end{longtable}

\begin{prop}\label{PropEvident}
For each $k\in \{6,\ldots, 18\},$
let $\ell_k$ be the number of  vertices of the Fano polytope  $P_k$,
 $L_k$ be the lattice given in Table 1
 and $S_k$ be a very general member of the family $\mathscr{F}_k$.
Then, there is a sublattice $E_k$ of ${\rm NS}(S_k)$,
which we call the evident lattice of  $S_k$,
satisfying ${\rm rank} (E_k) =23 - \ell_k$ and $|\det (E_k)| = |\det(L_k)|.$
The generators of $E_k$ are explicitly  given in Table 6.
\end{prop}

\begin{longtable}{llll}
  \caption{\normalsize{Evident lattices $E_k$ }}
  \vspace{-5.5mm}\\
  \\
\hline
  $k$ & $E_k$ & ${\rm rank}(E_k)$ & $|\det(E_k)|=|\det (L_k)|$   
    \\
  \hline
  \endhead
\vspace{1mm}$6$ &$\langle F, (O), (Q), (O'),  a_2 ,\ldots, a_7,b_1,\ldots,b_7 \rangle_\mathbb{Z}$ & \hspace{3.9mm} $17$ &\hspace{12.9mm} $ 16$   \\
\vspace{1mm} $7$ &  $\langle F, (O), (Q), a_1 ,\ldots, a_7,b_1,\ldots,b_7 \rangle_\mathbb{Z}$ &\hspace{3.9mm} $17$ &\hspace{12.9mm} $12$\\
\vspace{1mm}$8$  & $\langle F, (O), (Q), a_1 ,\ldots, a_5,b_1, b_2, c_1\ldots,c_7 \rangle_\mathbb{Z}$  &\hspace{3.9mm} $17$ &\hspace{12.9mm} $20$   \\
\vspace{1mm}$9$ & $\langle F, (O), (Q), (O'), a_1 ,\ldots,a_5, b_1,\ldots ,b_8 \rangle_\mathbb{Z}$ &\hspace{3.9mm} $17$  &\hspace{12.9mm} $16$    \\
\vspace{1mm}$10$ &$\langle F, (O),(Q), a_1 ,\ldots, a_4,b_1, \ldots,b_{10} \rangle_\mathbb{Z}$&\hspace{3.9mm} $17$  &\hspace{12.9mm} $14$   \\
\vspace{1mm}$11$ &$\langle F, (O),  (Q),a_1 ,\ldots, a_6,b_1, \ldots,b_{8} \rangle_\mathbb{Z}$  &\hspace{3.9mm} $17$ &\hspace{12.9mm} $12$  \\
\vspace{1mm}$12$ &$\langle F, (O), (Q), a_2 ,\ldots, a_5, b_1,b_2,  c_1\ldots,c_7 \rangle_\mathbb{Z}$ &\hspace{3.9mm} $17$ &\hspace{12.9mm} $18$   \\
\vspace{1mm}$13$ &$\langle F, (O), (Q),(O'), a_1 ,\ldots,a_4, b_1, c_1,\ldots ,c_7 \rangle_\mathbb{Z} $ &\hspace{3.9mm} $16$ &\hspace{12.9mm} $28$    \\
\vspace{1mm}$14$ &$\langle F, (O),  (Q),a_1 ,\ldots, a_6,b_1, \ldots,b_{7} \rangle_\mathbb{Z}$ &\hspace{3.9mm} $16$ &\hspace{12.9mm} $23$    \\
\vspace{1mm}$15$ &$\langle F, (O),(Q), a_1 ,\ldots,a_4, b_1,b_2, c_1,\ldots ,c_7 \rangle_\mathbb{Z}$ &\hspace{3.9mm} $16$ &\hspace{12.9mm} $31$   \\
\vspace{1mm}$16$ &$\langle F, (O),  (Q),a_1 ,\ldots, a_4,b_1, \ldots,b_9 \rangle_\mathbb{Z}$ &\hspace{3.9mm} $16$  &\hspace{12.9mm} $20$  \\
\vspace{1mm}$17$ &$\langle F, (O),  (Q), (O'), a_1 ,\ldots, a_4,b_1,c_1,d_1, \ldots,d_5 \rangle_\mathbb{Z}$ &\hspace{3.9mm} $15$ &\hspace{12.9mm} $48$   \\
\vspace{1mm}    $18$ & $\langle F, (O),(Q), a_1 ,\ldots,a_6, b_1,b_2, c_1,\ldots ,c_4 \rangle_\mathbb{Z}$ &\hspace{3.9mm} $15$ &\hspace{12.9mm} $44$  \\
\hline\\
\end{longtable} 

\begin{proof}
By applying the theory of \cite{Shioda}, we have the assertion.
\end{proof}

 \section{Period mappings and Picard numbers}
 
 In this section, we will see the rank of ${\rm NS} (S_k)$, where $S_k$ is a very general member of the family $\mathscr{F}_k$ of (\ref{Family}), is equal to the rank of $E_k$ in Proposition \ref{PropEvident}.
Our argument is based on the properties of  elliptic fibrations and period mappings of $K3$ surfaces.

For $k\in \{6,\ldots, 18\},$
take  $\lambda^0 =(\lambda_1^0,\ldots, \lambda_{l_k-3}^0) \in (\mathbb{C}^\times)^{\ell_k-3}$
such that $S_k^0 = S_k (\lambda^0)$ is a very general member of $\mathscr{F}_k$.
We call $S_k^0=S_k (\lambda^0)$  a reference surface of  $\mathscr{F}_k$.
Let us introduce a marking for $S_k(\lambda^0)$ using the evident lattice $E_k$ in Proposition \ref{PropEvident}.
According to the procedure,
$E_k$ is a sublattice of the N\'eron-Severi lattice ${\rm NS}(S_k^0)$.
Let $\widehat{E}_k$ be the primitive closure of $E_k$ in ${\rm NS}(S_k^0)$:
$\widehat{E}_k = (E_k \otimes_\mathbb{Z} \mathbb{Q}) \cap {\rm NS}(S_k^0)$.
Take a basis $\{\Gamma_{k,1},\ldots, \Gamma_{k,23 -\ell_k}\}$ of $\widehat{E}_k$. 
 Let us identify $H_2 (S_k^0,\mathbb{Z}) $ with  $L_{K3}$.
 This identification gives an isometry
 $\psi^0: H_2 (S_k^0,\mathbb{Z}) \simeq L_{K3}$.
 Put $\mathcal{E}_k = \psi^0 \left(\widehat{E}_k \right)$ and $\gamma_{k,j}  =\psi \left(\Gamma_{k,j} \right)$ ($j \in\{1,\ldots, 23-\ell_k\}$).
Then, $\mathcal{E}_k$ satisfies $\left(\psi^0\right)^{-1} (\mathcal{E}_k) \subset{\rm NS}(S_k^0)$.
We can  extend $\left\{\gamma_{k,1},\ldots, \gamma_{k,23-\ell_k}\right\}$ to  a basis $\left\{\gamma_{k,1},\ldots, \gamma_{k,22} \right\}$ of $L_{K3}$.
 Let $\left\{ \delta_{k,1},\ldots, \delta_{k,22}  \right\}$ be the dual basis of $\left\{\gamma_{k,1},\ldots, \gamma_{k,22}\right\}$ with respect to the unimodular lattice $L_{K3}$.

If we take a sufficiently small neighborhood $\mathcal{U}$ of $\lambda^0$, we have the  trivialization
$\tau: \{S_k(\lambda)\mid\lambda\in \mathcal{U}\} \rightarrow S_k^0 \times \mathcal{U}$.
Letting $\beta: S_k^0 \times \mathcal{U} \rightarrow S_k^0$ be the projection, 
we set $r=\beta \circ \tau.$
Then, for $\lambda\in \mathcal{U}$,
$\left(r^{(\lambda)}_k \right) ' := r|_{S_k(\lambda)}: S_k(\lambda)\rightarrow S_k^0 $ gives a $\mathcal{C}^\infty $-isomorphism of compact complex surfaces.
We obtain an isometry
$\psi^{(\lambda)}= \psi^0 \circ \left(r^{(\lambda)}_k \right) '_{*} : H_2(S_k(\lambda),\mathbb{Z}) \rightarrow L_{K3}$ for $\lambda \in \mathcal{U}$. 
We call the pair $(S_k(\lambda),\psi^{(\lambda)})$ an $S$-marked $K3$ surface on $\mathcal{U}$.
For $\lambda^1$ and $\lambda^2\in \mathcal{U}$,
if there exists a biholomorphic mapping $f : S_k (\lambda^1) \rightarrow S_k(\lambda^2)$ satisfying
$\left( \psi^{\left( \lambda^2 \right)} \circ f_* \circ \left(\psi^{\left( \lambda^1 \right)} \right)^{-1} \right)\Big|_{\mathcal{E}_k}$, 
we say that
$\left( S_k (\lambda^1),\psi^{\left( \lambda^1 \right)} \right)$ and $\left( S_k (\lambda^2), \psi^{\left( \lambda^2 \right)} \right)$ are equivalent as $S$-marked $K3$ surfaces.

Letting the notation be as above, we denote the intersection matrix of the lattice $\left\{ \delta_{k,23-\ell_k+1 },\ldots, \delta_{k,22}\right\}$ by $\textbf{A}_k$.
Set 
$\mathcal{D}_{\mathcal{E}_k}
=\{ \xi \in \mathbb{P}^{\ell_k-2} (\mathbb{C})\mid
\xi \textbf{A}_k {}^t \xi=0, \xi \textbf{A}_k {}^t \overline{\xi}>0\}$.
We remark that $\mathcal{D}_{\mathcal{E}_k}$ has two connected components.
Let $\mathcal{D}_k$ be a connected component of $\mathcal{D}_{\mathcal{E}_k}$.
Then, we have the local period mapping
$\Phi_k : \mathcal{U} \rightarrow \mathcal{D}_k$ given by
\begin{align}\label{PhiInt}
\lambda \mapsto \left(\int_{(\psi^{(\lambda)})^{-1}\left(\gamma_{k,23-\ell_k +1} \right) }  \omega_k : \cdots  :  \int_{(\psi^{(\lambda)})^{-1}\left(\gamma_{k,22} \right) }  \omega_k\right).
\end{align}
Here, $\omega_k$ is the unique holomorphic $2$-form on $S_k$ up to a constant factor.

Let $\pi: S\rightarrow \mathbb{P}^1 (\mathbb{C})$ be an elliptic fibration of a $K3$ surface $S$ with a general fibre $F$. 
If  $\pi$ is defined by the Weierstrass equation
$z^2 =4 y^3 -g_2(x) y - g_3(x)$, where $\pi $ is given by $(x,y,z)\mapsto x$,
then the $j$-invariant is given by
$j(x)=\frac{g_2^3(x)}{g_2^3(x)-27 g_3^2(x)}$.
Suppose $(S,\pi,\mathbb{P}^1(\mathbb{C}))$ and $(S',\pi',\mathbb{P}^1(\mathbb{C}))$ are two elliptic surfaces 
such that there exists a biholomorphic mapping $f: S \rightarrow S'$ and $\varphi \in {\rm Aut}(\mathbb{P}^1(\mathbb{C}))$
satisfying $\varphi \circ \pi = \pi' \circ f$.
Then, they are said to be isomorphic as elliptic fibrations.
In this case,
if $(S,\pi,\mathbb{P}^1(\mathbb{C}))$ ($(S',\pi',\mathbb{P}^1(\mathbb{C}))$, resp.) is defined by the Weierstrass equation with the $j$-invariant $j (x)$ ($j' (x)$, resp.),
then $j' \circ \varphi  = j$ holds (see \cite{Kod} or \cite{Ne}).
Moreover, 
if $\pi^{-1} (x)$ is a singular fibre,
then  $\pi'^{-1} (\varphi (x))$ is a singular fibre of the same type.

\begin{lem}\label{LemPsi}
For each $k\in \{6,\ldots,18\}$, let $S_k (\lambda^0)$ be a very general member of $\mathscr{F}_k$ and $\mathcal{U} $ be a small neighborhood of $\lambda^0$ as above.
Let $\pi_k^{(\lambda)}$ be the Jacobian elliptic fibration given by the equation (\ref{EquationJac}).
Take $\lambda^1$ and $\lambda^2 \in \mathcal{U}$.
If $\left(S_k(\lambda^1), \pi_k^{\left( \lambda^1 \right)}, \mathbb{P}^1 (\mathbb{C}) \right)$ and  $\left(S_k(\lambda^2), \pi_k^{\left( \lambda^2 \right)}, \mathbb{P}^1 (\mathbb{C}) \right)$ are isomorphic as elliptic surfaces,
then $\lambda^1 = \lambda^2$ holds.
\end{lem}

\begin{proof}
By observing the Kodaira type of the singular fibres, together with the above mentioned  properties of the $j$-invariant for each $\pi_k$, 
we can prove this lemma via a similar argument to the proof of \cite{NaP} Lemma 1.1. 
\end{proof}

\begin{lem}\label{LemSEquiv}
For $k\in \{6,\ldots,18\}$, let $S_k (\lambda^0)$ be a very general member of $\mathscr{F}_k$ and $\mathcal{U} $ be a small neighborhood of $\lambda^0$ as above.
 Take the $S$-marked $K3$ surface $\left(S_k(\lambda), \psi^{(\lambda)} \right)$ for  $\lambda \in \mathcal{U}$ as above.
For $\lambda^1$ and $\lambda^2 \in \mathcal{U}$,
$\left(S_k(\lambda^1),\psi^{\left( \lambda^1 \right)} \right)$ and $\left(S_k(\lambda^2),\psi^{\left( \lambda^2 \right)} \right)$ are equivalent as $S$-marked $K3$ surfaces if and only if   $\lambda^1 = \lambda^2$ holds.
\end{lem}

\begin{proof}
For each $k$,
let $F_k^{\left( \lambda^j \right)}$ $(j\in \{ 1,2 \})$ be a general fibre for the elliptic fibration $\pi_k^{\left( \lambda^j \right)}$.
Let  
$f: S_k(\lambda^1) \rightarrow S_k(\lambda^2)$
a biholomorphic mapping 
which induces an equivalence of $S$-marked $K3$ surfaces.
Then,   $f_* \left(F_k^{\left( \lambda^1 \right)}\right) = F_k^{\left( \lambda^2 \right)}$ holds.
Therefore, 
by applying Lemma \ref{LemPsi},
we can obtain the assertion
by a similar argument to the proof of \cite{NaP} Lemma 1.2.
\end{proof}

According to Lemma \ref{LemSEquiv},
we can apply  the Torelli theorem to our local period mapping $\Phi_k$ of (\ref{PhiInt}).
Then, we can prove the following theorem as in the proof of \cite{NaP} Theorem 1.1.

 \begin{thm}\label{ThmPic}
For each $k\in \{6,\ldots,18\},$
the Picard number of a  very general member of  $\mathscr{F}_k$ of (\ref{Family}) 
is equal to ${\rm rank}(E_k)=23-\ell_{k}$, where $E_k$ is the evident lattice in Proposition \ref{PropEvident}.
\end{thm}

\section{Lattice structures}

Let $L$ be a non-degenerate even lattice.
Set $L^\vee ={\rm Hom}(L,\mathbb{Z}).$
The group $\mathscr{A}_L = L^\vee /L$ becomes to be a finite abelian group.
The length $l(\mathscr{A}_L)$ of $\mathscr{A}_L$ is defined as the minimum number of generators of $\mathscr{A}_L.$
The intersection form of $L$
induces a quadratic form $q_L:\mathscr{A}_L \rightarrow \mathbb{Q}/2\mathbb{Z}.$
This $q_L$ is called the discriminant form of $L$.
We note that $q_{L_1 \oplus L_2} \simeq q_{L_1} \oplus q_{L_2}$ holds.
If the signature of $L$ is of $(s,t)$,
the triple $(s,t,q_L)$ is called the invariant of $L$.

\begin{prop} (\cite{Ni} Proposition 1.6.1)\label{PropLatticeOrthogonal}
Let $L$ be a non-degenerate even lattice.
Suppose $L$ is a primitive lattice in a unimodular lattice $\widehat{L}$.
Letting $L^\perp$ be the orthogonal complement of $L$ in $\widehat{L}$,
one has $q_{L^\perp} \simeq -q_{L}$.
 \end{prop}

\begin{prop} (\cite{Ni} Corollary 1.13.1) \label{PropQI}
Let $L$ be a non-degenerate even lattice with the invariant $(s,t,q_L).$
Suppose $s>0$, $t>0$ and $l(\mathscr{A}_L)\leq {\rm rank}(L)-2$ hold.
Then, $L$ is the unique lattice with the invariant $(s,t,q_L)$ up to isometry.
\end{prop}

We will apply these arithmetic  results for lattices to our N\'eron-Severi lattices of toric $K3$ hypersurfaces.

\begin{lem}\label{LemFinAbel}
For $k\in\{6,\ldots, 18\},$ let $L_k$ be the lattice of Table 1 and $E_k$ be the evident lattice of Proposition \ref{PropEvident}.
The finite abelian group $\mathscr{A}_{E_k}$ is isomorphic to the group $\mathscr{A}_{L_k}.$ 
Precisely,  setting  
\begin{align*}
& 
\mathcal{G}_6=(\mathbb{Z}/4\mathbb{Z})\oplus (\mathbb{Z}/2\mathbb{Z})^{ \oplus 2} ,\quad 
\mathcal{G}_7=\mathbb{Z}/12 \mathbb{Z}, \quad
\mathcal{G}_8 = \mathbb{Z}/10 \mathbb{Z}, \quad 
\mathcal{G}_9 =\mathbb{Z}/16 \mathbb{Z}, \quad
\mathcal{G}_{10}=\mathbb{Z}/14 \mathbb{Z},
\\
&
\mathcal{G}_{11}=\mathbb{Z}/12 \mathbb{Z}, \quad
\mathcal{G}_{12}=\mathbb{Z}/18 \mathbb{Z}, \quad
\mathcal{G}_{13}=(\mathbb{Z}/14 \mathbb{Z})\oplus (\mathbb{Z}/2 \mathbb{Z}), \quad
\mathcal{G}_{14}=\mathbb{Z}/23 \mathbb{Z},\quad
\mathcal{G}_{15}=\mathbb{Z}/31 \mathbb{Z},\\
&
\mathcal{G}_{16}=(\mathbb{Z}/10\mathbb{Z})\oplus (\mathbb{Z}/2 \mathbb{Z}),\quad
\mathcal{G}_{17}=(\mathbb{Z}/12 \mathbb{Z}) \oplus (\mathbb{Z}/2 \mathbb{Z})^{\oplus 2}, \quad
\mathcal{G}_{18}=\mathbb{Z}/44 \mathbb{Z},
\end{align*}
one has $\mathscr{A}_{E_k} \simeq \mathscr{A}_{L_k} \simeq \mathcal{G}_k$ for each $k$.
Also,  $q_{E_k} \simeq - q_{L_k}$ holds.
\end{lem}

\begin{proof}
For each $k\in\{6,\ldots, 18\},$
let $\{v_1,\ldots,v_{\ell_k-3}\}$ ($\{w_1,\ldots,w_{23-\ell_k}\}$,  resp.)  be the basis of $L_k$  ($E_k$, resp.)
of the intersection matrix of Table 1 
(Table 6, resp.),
where $\ell_k$ is the number of vertices of $P_k$.
Then, we can find an element 
$\alpha \in L_k^\vee / L_k$ 
($\beta\in E_k^\vee/E_k$, resp.) 
in the form
$\alpha = \sum_{i=1}^{\ell_k-3} r_i v_i $
($\beta = \sum_{i=1}^{23-\ell_k} s_i w_i $),
where $r_i$ ($s_i$, resp.) are elements of $\mathbb{Q}/\mathbb{Z}$,
such that $\alpha $ ($\beta$, resp.) generates a cyclic group which is a direct summand of $\mathcal{G}_k$.
One can find the coefficients $r_i$ ($s_i$, resp.) in Table A.2.1 (Table A.2.2, resp.) of Appendix.
The values $q_{L_k}(\alpha)$ and $q_{E_k}(\beta)$ for each $\alpha$ and $\beta$ are listed in Appendix.
Thus, we can check $q_{E_k} \simeq - q_{L_k}$ for every $k\in \{6,\ldots, 18\}$.
\end{proof}

\section{Mordell-Weil groups for our elliptic fibrations}

 Let  $\pi : S \rightarrow C$ be a Jacobian elliptic fibration.
We assume that  $\pi$ has singular fibres.
We let ${\rm MW}(\pi,O)$ denote the Mordell-Weil group of sections of $\pi$.
For all $P\in {\rm MW}(\pi,O)$ and $v\in C$, 
we have $(P\cdot \pi^{-1}(v))=1$.
Note that the section $P$ intersects  an irreducible component with multiplicity $1$ of every fibre $\pi^{-1}(v)$.
Set 
$
\mathcal{R} = \{ v\in \mathbb{C} \mid  \pi^{-1}(v) {\rm \enspace is \enspace a \enspace singular \enspace fibre \enspace of \enspace} \pi \}.
$
For all $v\in \mathcal{R}$, we have an expression
$
\pi^{-1}(v)=\Theta_{v,0}+ \sum_{j=1}^{m_v-1} \mu _{v,j} \Theta_{v,j},
$
where $m_v$ is the number of irreducible components of $\pi^{-1}(v)$, $\Theta_{v,j} \hspace{1mm} (j=0,\cdots,m_v -1)$ are irreducible components with multiplicity $\mu_{v,j}$ of $\pi^{-1}(v)$
and $\Theta_{v,0}$ is the component such that $\Theta_{v,0} \cap (O) \not= \phi$.
Let $F$ be a general fibre of $\pi$.
The lattice
$
T=\langle F, (O), \Theta_{v,j} \mid v\in \mathcal{R}, 1\leq j \leq m_v-1   \rangle_{\mathbb{Z}} \subset {\rm NS}(S)
$
is called  the trivial lattice for $\pi$.
Let $\widehat{T}$ be the primitive closure of $T$ in ${\rm NS}(S)$: $\widehat{T}=(T \otimes _{\mathbb{Z}}  \mathbb{Q}) \cap {\rm NS}(S)$.

\begin{prop}\label{propMW} (\cite{Shioda})
{\rm (1)} One has  ${\rm MW}(\pi,O) \simeq {\rm NS}(S)/T$  given by $P \mapsto (P) \hspace{1mm} ({\rm mod } \hspace{1mm} T)$.

{\rm (2)} 
The rank of ${\rm MW}(\pi,O)$ is equal to ${\rm rank} ({\rm NS}(S))-2- \sum_{v\in \mathcal{R}} (m_v -1).$

{\rm (3)}  
Let ${\rm MW}(\pi,O)_{tor}$ be the torsion part of ${\rm MW}(\pi,O)$. Then,
$
{\rm MW}(\pi,O)_{tor} \simeq \widehat{T}/T.
$
\end{prop}

For $v\in \mathcal{R}$, we set
$
(\pi^{-1} (v))^\sharp = \bigcup_{0\leq j \leq m_v -1, \mu_{v,j} =1} \Theta_{v,j}^\sharp,
$
where
$
\Theta_{v,j}^\sharp = \Theta_{v,j} -\{\text{singular points of } \pi^{-1} (v)\}.
$
Then, the set $(\pi^{-1} (v))^\sharp$ admits a group structure (\cite{Kod}, \cite{Ne}).
Precisely,
\begin{align}\label{NeronHom}
(\pi^{-1} (v))^\sharp \simeq
 \mathbb{C}^\times \times (\mathbb{Z}/b \mathbb{Z}) \quad (\text{if $\pi^{-1}(v)$ is of type $I_b$}).
 \end{align}
Letting $v\in \mathcal{R}$, we have the mapping $sp_v: {\rm MW}(\pi,O) \rightarrow (\pi^{-1} (v))^\sharp$ defined by $P \mapsto P \cap \pi^{-1}(v).$
This mapping, called the specialization mapping, gives a homomorphism of groups.
Setting ${\rm MW}(\pi,O)_0 =\{P\in {\rm MW}(\pi,O) \mid (P) \cap \Theta_{v,0} \not=\phi \text{ for all } v\in \mathcal{R}\}$, we have ${\rm MW}(\pi,O)_0 \subset {\rm MW}(\pi,O)/ {\rm MW}(\pi,O)_{tor}.$

We will apply the above result to our elliptic $K3$ surfaces.

\subsection{$E_k$ as N\'eron-Severi lattice}

For each $k\in \{6, \ldots, 18\},$
let $S_k$ be a very general member of the family $\mathscr{F}_k$ of (\ref{Family})
and $\pi_k$ be the elliptic fibration given in Proposition \ref{PropEquationJac}.
Let $T_k$ be the trivial lattice for $\pi_k$,
which is coming from the  fibres  and the sections  illustrated in Table 5.

 \begin{thm}\label{ThmNSEvident}
The N\'eron-Severi lattice ${\rm NS}(S_k)$ is isometric to the evident lattice $E_k$ of Proposition \ref{PropEvident}.
\end{thm}

We will prove the theorem for each $k$.
Note  that  ${\rm rank}(E_k) = {\rm rank}({\rm NS}(S_k))$ is proved in Theorem \ref{ThmPic}.

\subsubsection{Cases of $k=10,14,15$}

If $|{\rm det} (E_k)|$ is square-free,
then $E_k$ is equal to ${\rm NS}(S_k)$.
Therefore, from Table 6, Theorem \ref{ThmNSEvident} is true if $k=10,14,15$.

\subsubsection{Cases of $k=7,8,11,18$}

Suppose $\frac{|{\det} (E_k)|}{2^2}$ is an odd square-free number.
Then, $\left[{\rm NS}(S_k): E_k \right] =1 \text{ or } 2$.
If $\left[{\rm NS}(S_k): E_k \right] = 2$, then ${\rm NS}(S_k)$ is an even lattice whose discriminant is odd.
However, this is a contradiction if the rank of ${\rm NS}(S_k)$ is odd.
Thus, Theorem \ref{ThmNSEvident} holds for $k=7,8,11,18$.

\subsubsection{Case of $k= 16$}

As in Section 5.1.2, we have $\left[{\rm NS}(S_{16}): E_{16} \right] =1 \text{ or } 2$.

\begin{lem}\label{LemCase1.1}
One has $\left[\widehat{T}_{16} : T_{16} \right]=1$.
\end{lem}

\begin{proof}
Assume $\left[\widehat{T}_{16} : T_{16} \right]= 2$. 
Then, due to Proposition \ref{propMW} (3), there exists $R\in {\rm MW}(\pi_{16}, O)_{tor}$ such that $2 R =O$.
By virtue of the structure of (\ref{NeronHom}) and the property of the specialization mapping,
we can observe how the section $(R)$ intersects the divisors $a_p, b_q$ illustrated in Table 5.
Let $T_{16}' $ be the group generated by $(R)$ and  $T_{16}$.
We can see that ${\rm det} (T_{16}' )$  never become zero (for precise calculations, see Appendix Section A.3).
This is a contradiction, because  ${\rm rank} (T_{16} ') = {\rm rank} (T_{16})$.
Hence, we have the assertion.
\end{proof}

\begin{lem}\label{LemCase1.2}
One has   $\left[{\rm NS}(S_{16}): E_{16} \right] =1$.
\end{lem}

\begin{proof}
According to Lemma \ref{LemCase1.1}, 
${\rm NS}(S_{16}) $
is equal to the primitive closure of the group generated by
$T_{16}$ and $(Q)$ in ${\rm NS}(S_{16}) $.
Suppose there exists $Q'\in {\rm MW}(\pi_{16}, O)$ with $2 Q' =Q$.
Let $\widetilde{E}_{16}$ be the group generated by $E_{16}$ and $(Q').$
We can observe how $Q'$ intersects $T_{16}$.
Then, we have ${\rm det}\left( \widetilde{E}_{16} \right)\not=0 $  (for detail, see Appendix Section A.3).
This is a contradiction, because ${\rm rank}(E_{16}) = {\rm rank} \left(\widetilde{E}_{16} \right)$.
Thus, $\left[ {\rm NS}(S_{16}) : E_{16} \right]=1.$
\end{proof}

\subsubsection{Cases of $k= 9,  13, 17$}

If  $k\in \{9,13,17\},$ there is the $2$-torsion element $O' \in {\rm MW}(\pi_k, O)_{tor}$ in Table 4.
Set $E_k' = T_k + \langle (Q) \rangle_\mathbb{Z} $.
It holds ${\rm rank} (E_k') = {\rm rank} ({\rm NS}(S_k))$.
By calculating $|{\rm det} (E_k')|,$ we can see that $\left[{\rm NS}(S_k): E_k' \right] =2 \text{ or } 4 \text{ or } 8$.
By an argument similar to the proof of  Lemma \ref{LemCase1.1}, we can obtain the following lemma  
(see the precise data of Appendix Section A.3).

\begin{lem}\label{LemCase2.1}
If $k\in \{9, 13, 17\}$, $\left[\widehat{T}_k : T_k\right]=2$ holds.
\end{lem}

Now, let $\overline{T}_k$ be the lattice generated by the generators of $E_k$ of Table 6 with the exception of $(Q)$.
Remark that $(O')\in \overline{T}_k $.
Since $(O') \in  \widehat{T}_k$, $\overline{T}_k $ is a sublattice of $\widehat{T}_k$.
By calculating ${\rm det}\left( \overline{T}_k \right)$ and ${\rm det} \left( \widehat{T}_k \right)$, together with Lemma \ref{LemCase2.1},  we obtain the following lemma.

\begin{lem}\label{LemCase2.2}
If $k\in \{9, 13, 17\}$, $\left[\widehat{T}_k : \overline{T}_k\right]=1$ holds.
\end{lem}
 
By an argument like Section 5.1.2, we have $[{\rm NS}(S_k):E_k]=1 \text{ or } 2$.
According to (\ref{NeronHom}), there are no sections $Q'$ with $2Q'=Q$.
Assume there exists a section $Q'$ such that $2Q' = Q+O'$.
Then, there exist $Q'+O'$ and $-Q'$ also.
For $Q_0' \in \{Q', Q'+O', -Q' \}$, let us consider the group $\check{E}_k = \overline{T}_k +\langle (Q'_0) \rangle_\mathbb{Z}$, which should give an overlattice of $E_k$.
However, by calculating ${\rm det}\left( \check{E}_k \right)$ for $Q_0' \in \{Q', Q'+O', -Q' \}$, 
we can see  it is impossible that   $\overline{T}_k +\langle (Q')\rangle_\mathbb{Z} $  becomes  a proper overlattice of $E_k$,
because ${\rm det}\left( \check{E}_k \right)$  take inappropriate values.
For detailed calculations, see Appendix Section A.3.
Therefore, $[{\rm NS}(S_k):E_k] \not= 2$. 
Thus, we have the following lemma.
 
\begin{lem}\label{LemCase2.3}
If $k\in \{9,13,17\}$,   $\left[{\rm NS}(S_k): E_k \right] =1$ holds.
\end{lem}

\subsubsection{Case of $k=6$}

Since there exists the 2-torsion element $O'\in {\rm MW}(\pi_6,O)_{tor}$ in Table 4,  together with $|{\rm det}(T_{6})| = 2^6$, $\left[\widehat{T}_6 : T_6 \right]$  must be $2 \text{ or } 4 \text{ or } 8.$
We can prove $\left[\widehat{T}_6:T_6\right]=2$ as in the proof of Lemma \ref{LemCase1.1}.
Taking  the sublattice $\overline{T}_6 $ of $\widehat{T}_6$ such that $E_6=\widehat{T}_6 +\langle (Q) \rangle_\mathbb{Z}$ like the above $\widehat{T}_{k}$ $(k\in \{9,13,17\})$,
 we can see $\left[\widehat{T}_6 : \overline{T}_6\right] =1  $ as in Lemma \ref{LemCase2.2}.
As in Section 5.1.2, we have $[{\rm NS}(S_6):E_6]=1 \text{ or } 2$.
Suppose there exists a section $Q'$ with $2 Q' = Q.$
Setting $\widetilde{E}_6 = E_6 + \langle (Q') \rangle_\mathbb{Z}$, 
as in the proof of Lemma \ref{LemCase1.2}, we can see that ${\rm det}\left(\widetilde{E}_6\right)$ never become zero.
See the precise data of Appendix Section A.3.
So, there does not exist $Q'$ with $2Q'=Q$.
Since the configuration  for $k=6$ in Table 5 is symmetric,  it guarantees the nonexistence of $Q'$ with $2Q' =Q +O'$.
Thus, we have $[{\rm NS}(S_6): E_6]=1$.

\subsubsection{Case of $k=12$}

The proof for the case of $k=12$ is rather simple.
We have $\widehat{T}_{12} = {\rm NS}(S_{12})$, because ${\rm rank}(T_{12})={\rm rank}({\rm NS}(S_{12}))$.
Since $|{\rm det}(T_{12})| = 2 \cdot 3^4$, $\left[{\rm NS}(S_{12}): T_{12} \right]= 1 \text{ or } 3 \text{ or } 9.$ 
The section $Q$ satisfies $3Q=O$.
We can observe that 
${\rm MW}(\pi_{12},O)_{tor}$ is just generated by $Q$
 on the basis of (\ref{NeronHom}) and the divisors illustrated in Table 5.
Hence, we have $[{\rm NS}(S_{12}): T_{12}]=3$. 
Therefore, we have ${|{\rm det}(E_{12})|}=2 \cdot 3^2= |{\rm det}({\rm NS}(S_{12}))| $ and  $E_{12} = {\rm NS}(S_{12})$.

\subsection{Consequences of Theorem \ref{ThmNSEvident}}

 By Theorem \ref{ThmNSEvident}, together with Proposition \ref{PropLatticeOrthogonal}, Proposition \ref{PropQI} and Lemma \ref{LemFinAbel}, we have the following Corollary.

 \begin{cor}\label{CorTr}
For each $k\in \{6,\ldots,18\},$
let $S_k$ be a very general member of the family $\mathscr{F}_k$ of (\ref{Family})
and $L_k$ be the lattice in Table 1.
Then, the transcendental lattice ${\rm Tr}(S_k)$ is isometric to $U\oplus L_k$.
\end{cor}

Hence, according to (\ref{K3mirrorF}),
the Dolgachev conjecture is established for every three-dimensional Fano polytope.

 \begin{rem}\label{RemHes}
 The lattice structure of  $S_6$ is determined in \cite{DG} Proposition 5.4.
 In \cite{DG}, they study $K3$ surfaces, 
 which are coming from the Hessian of a general cubic surfaces which does not admit a Sylvester form.
 Such Hessian $K3$ surface is birationally equivalent to our $S_6$
 (see  \cite{KHes} also).
 \end{rem}

From the argument  of Section 5.1, 
together with Proposition \ref{propMW}, 
we obtain the structure of the Mordell-Weil group ${\rm MW}(\pi_k, O)$ for each $k$.

\begin{thm}\label{ThmMW}
Let $S_k$ be a very general member of the family $\mathscr{F}_k$ of (\ref{Family})
and $\pi_k$ be the elliptic fibration given in Proposition \ref{PropEquationJac}.

(1) If $k\in \{7,8,10,11,14,15,16,18\}$,
the rank of ${\rm MW}(\pi_k, O)$ is one. 
The free part of ${\rm MW}(\pi_k, O)$ is generated by $Q$ in Table 4. 
Also, ${\rm MW}(\pi_k, O)_{tor}=\{O\}.$

(2) If $k\in \{6,9,13,17\}$,
the rank of ${\rm MW}(\pi_k, O)$ is one. 
The free part of ${\rm MW}(\pi_k, O)$ is generated by $Q$ in Table 4. 
Also, $O'$ in Table 4 generates ${\rm MW}(\pi_k, O)_{tor} \simeq \mathbb{Z}/2\mathbb{Z}$. 

(3) If $k=12$,
the rank of ${\rm MW}(\pi_k, O)$ is zero. 
Also, $Q$ in Table 4 generates ${\rm MW}(\pi_k, O)_{tor}\simeq \mathbb{Z}/3\mathbb{Z}$. 
\end{thm}

 \section*{Appendix : Supplementary data}

 \subsection*{A.1 Tables  for  proof of Proposition 2.1}

By performing the birational transformation $(x,y,z)\mapsto (x_1,y_1,z_1)$ given by
$$
x=x(x_1,y_1,z_1), \quad y=y(x_1,y_1,z_1), \quad z=z(x_1,y_1,z_1),
$$
as in Table A.1.1, A.1.2 and A.1.3 to the equations of hypersurfaces in Table 2 
for each $k \in \{6,\ldots,18\}$,
we obtain elliptic $K3$ surfaces of Proposition 2.1.

{\small
\begin{longtable}{llll}
\caption*{Table A.1.1: Rational functions  $ x(x_1,y_1,z_1)$ }
 \vspace{-5.5mm}\\
 \\
%{\normalsize{Birational transformation  $(x,y,z)\mapsto (x_1,y_1,z_1)$ }}\\
\hline
  $k$ &  $x(x_1,y_1,z_1) $     
  %\vspace{1mm}
    \\
  \hline
  %\vspace{.1mm}
  \endhead
\vspace{1mm}$6$ &
 $\frac{2 y_1 (-\lambda_3 x_1^2 + y_1)}{x_1 (\lambda_2 y_1 + x_1 y_1 + x_1^2 y_1 + z_1)}$ 
  \\
\vspace{1mm}$7$ &
 $\frac{2 y_1^2}{x_1 (-\lambda_3 x_1^3 + x_1 y_1 + x_1^2 y_1 - z_1)}$  \\
\vspace{1mm}$8$ &
 $\frac{2 y_1^2}{x_1 (-\lambda_2 \lambda_3 x_1^2 - \lambda_3 x_1^3 + x_1 y_1 + x_1^2 y_1 + z_1)}$  \\
\vspace{1mm}$9$ &
 $\frac{2 y_1 (-\lambda_3 x_1 + y_1)}{x_1 (\lambda_2 y_1 + x_1 y_1 + x_1^2 y_1 - z_1)}$ \\
\vspace{1mm}$10$ &
 $\frac{2 y_1^2}{x_1 (-\lambda_3 x_1^2  + \lambda_1 y_1 + x_1 y_1 + x_1^2 y_1+z_1)}$ 
  \\
\vspace{1mm}$11$ &
 $\frac{2 y_1^2}{x_1 (-\lambda_3 x_1^2 + x_1 y_1 + x_1^2 y_1 + z_1)}$ 
 \\
\vspace{1mm}$12$ &
 $\frac{2 y_1^2}{x_1 (-\lambda_2 \lambda_3 x_1^2 - \lambda_3 x_1^3 + \lambda_1 y_1 + x_1 y_1 + x_1^2 y_1 + z_1)}$ 
 \\
\vspace{1mm}$13$ & $\frac{2 y_1 (-\lambda_4 x_1 - \lambda_3 x_1^2 + y_1)}{x_1 (\lambda_2 y_1 + x_1 y_1 + x_1^2 y_1 - z_1)}$
\\
\vspace{1mm}$14$ & $\frac{2 (\lambda_4 x_1 + \lambda_3 x_1^2 - y_1) y_1}{x_1 (\lambda_1 x_1^3 - \lambda_2 y_1 - x_1 y_1 - x_1^2 y_1 -
    z_1)}$ 
  \\
\vspace{1mm}$15$ &
 $\frac{(\lambda_4 + \lambda_3 x_1)( -\lambda_1 \lambda_4 x_1^2 - \lambda_1 \lambda_3 x_1^3 + \lambda_2 y_1 + x_1 y_1 + x_1^2 y_1 + z_1)}{2 y_1 (-\lambda_4 x_1 - \lambda_3 x_1^2 + y_1)}$ 
\\
 \vspace{1mm}$16$ &
  $\frac{2 y_1 (-\lambda_4 x_1 - \lambda_3 x_1^2 + y_1)}{x_1 (-\lambda_1 x_1^2 + \lambda_2 y_1 + x_1 y_1 + 
   x_1^2 y_1 + z_1)}$ 
   \\
  \vspace{1mm}$17$ &
   $\frac{2 y_1 (-\lambda_2 \lambda_4 x_1 - \lambda_2 x_1^2 - \lambda_4 \lambda_5 x_1^2 - \lambda_5 x_1^3 + y_1)}{(\lambda_4 + 
   x_1) (\lambda_3 y_1 + x_1 y_1 + x_1^2 y_1 + z_1)}$ 
   \\
 \vspace{1mm}   $18$ &  $x_1$ 
 \\
 \hline\\
\end{longtable} 
}
\vspace{-3mm}
{\small
\begin{longtable}{llll}
\caption*{Table A.1.2: Rational functions  $ y(x_1,y_1,z_1)$ }
\vspace{-5.5mm}\\
\\
%{\normalsize{Birational transformation  $(x,y,z)\mapsto (x_1,y_1,z_1)$ }}\\
\hline
  $k$ &  $y(x_1,y_1,z_1)$     
  %\vspace{1mm}
    \\
  \hline
  %\vspace{.1mm}
  \endhead
\vspace{1mm}$6$ 
& $x_1$ 
  \\
\vspace{1mm}$7$ 
& $x_1$ 
 \\
\vspace{1mm}$8$ &
  $x_1$ 
  \\
\vspace{1mm}$9$ &
  $x_1$ 
  \\
\vspace{1mm}$10$ &
  $-\frac{-\lambda_3 x_1^2  + \lambda_1 y_1 + x_1 y_1 + x_1^2 y_1+z_1}{2 x_1 y_1}$
   \\
\vspace{1mm}$11$ &
  $x_1$ 
  \\
\vspace{1mm}$12$ & 
 $\frac{\lambda_2 \lambda_3 x_1^2 + \lambda_3 x_1^3 - \lambda_1 y_1 - x_1 y_1 - x_1^2 y_1 - z_1}{2 (\lambda_2 + x_1) y_1}$ 
 \\
\vspace{1mm}$13$ & 
 $x_1$
 \\
\vspace{1mm}$14$ &
      $x_1$ 
      \\
\vspace{1mm}$15$ &
  $x_1$
   \\
 \vspace{1mm}$16$ & 
    $x_1$
    \\
  \vspace{1mm}$17$ & 
   $x_1$ 
  \\
 \vspace{1mm}   $18$ & 
  $\frac{\lambda_1 \lambda_5 x_1^3 + \lambda_1 x_1^4 - \lambda_2 y_1 - x_1 y_1 - x_1^2 y_1 + z_1}{2 (\lambda_5 + x_1) y_1}$ 
  \\
 \hline\\
\end{longtable} 
}
\vspace{-3mm}
{\small
\begin{longtable}{llll}
\caption*{Table A.1.3: Rational functions  $ z(x_1,y_1,z_1)$ }
\vspace{-5.5mm}\\
\\
%{\normalsize{Birational transformation  $(x,y,z)\mapsto (x_1,y_1,z_1)$ }}\\
\hline
  $k$ &  $z(x_1,y_1,z_1)$    
  %\vspace{1mm}
    \\
  \hline
  %\vspace{.1mm}
  \endhead
\vspace{1mm}$6$ &
  $- \frac{\lambda_2 y_1 + x_1 y_1 + x_1^2 y_1 + z_1}{2 x_1 (-\lambda_3 x_1^2 + y_1)}$  \\
\vspace{1mm}$7$ &
 $-\frac{-\lambda_3 x_1^3 + x_1 y_1 + x_1^2 y_1 - z_1}{2 x_1 y_1}$ \\
\vspace{1mm}$8$ &
  $\frac{\lambda_2 \lambda_3 x_1^2 + \lambda_3 x_1^3 - x_1 y_1 - x_1^2 y_1 - z_1}{2 (\lambda_2 + x_1) y_1}$ \\
\vspace{1mm}$9$ &
  $\frac{\lambda_2 y_1 + x_1 y_1 + x_1^2 y_1 - z_1}{2 x_1 (\lambda_3 x_1 - y_1)}$ \\
\vspace{1mm}$10$ & 
 $x_1$ \\
\vspace{1mm}$11$ &  
$-\frac{-\lambda_3 x_1^2 + x_1 y_1 + x_1^2 y_1 + z_1}{2 x_1 y_1}$\\
\vspace{1mm}$12$ & 
$x_1$\\
\vspace{1mm}$13$ & 
 $\frac{\lambda_2 y_1 + x_1 y_1 + x_1^2 y_1 - z_1}{2 x_1 (\lambda_4 x_1 + \lambda_3 x_1^2 - y_1)}$ \\
\vspace{1mm}$14$ 
    & $-\frac{\lambda_1 x_1^3 - \lambda_2 y_1 - x_1 y_1 - x_1^2 y_1 - z_1}{
 2 x_1 (\lambda_4 x_1 + \lambda_3 x_1^2 - y_1)}$\\
\vspace{1mm}$15$ 
&  
  $-\frac{-\lambda_1 \lambda_4 x_1^2 - \lambda_1 \lambda_3 x_1^3 + \lambda_2 y_1 + x_1 y_1 + x_1^2 y_1 + z_1}{
 2 x_1 (-\lambda_4 x_1 - \lambda_3 x_1^2 + y_1)}$\\
 \vspace{1mm}$16$ & 
   $\frac{-\lambda_1 x_1^2 + \lambda_2 y_1 + x_1 y_1 + 
 x_1^2 y_1 + z_1}{2 x_1 (\lambda_4 x_1 + \lambda_3 x_1^2 - y_1)}$\\
  \vspace{1mm}$17$ &
      $\frac{\lambda_3 y_1 + x_1 y_1 + 
 x_1^2 y_1 + z_1}{2 x_1 (\lambda_2 \lambda_4 x_1 + \lambda_2 x_1^2 + \lambda_4 \lambda_5 x_1^2 + \lambda_5 x_1^3 - y_1)}$\\
 \vspace{1mm}   $18$ & 
   $-\frac{\lambda_1 x_1^2 (\lambda_5 + x_1)}{y_1}$\\
 \hline\\
\end{longtable} 
}

 \subsection*{A.2 Tables  for  proof of Lemma 4.1}

In the proof of Lemma 4.1 of [MN],
 $\alpha$ and $\beta$ are necessary.
 Here, we give explicit forms of them.

For each $k\in \{6,\ldots,18\}$,
$$\alpha = \sum_{k=1}^{\ell_k-3} s_k v_k $$
is defined by the coefficients in Table A.2.1.

\begin{longtable}{llll}
\caption*{Table A.2.1: Coefficients  $r_1,\ldots,r_{\ell_k-3}$ and discriminants  $q_{L_k}(\alpha) $}
\vspace{-5.5mm}\\
\\
\hline
  $k$ & Cyclic groups &  Coefficients  $r_1,\ldots, r_{\ell_k-3}$  & $q_{L_k} (\alpha)$  \vspace{1mm}
  %\vspace{1mm}
    \\
  \hline   %\vspace{.1mm}
  \endhead
\vspace{1mm}$6$ &$ \mathbb{Z}/4\mathbb{Z}$ & $\frac{3}{4},\frac{1}{4},\frac{1}{4}$   & $\frac{7}{4}$ \vspace{1mm}\\
\vspace{1mm} &$\mathbb{Z}/2\mathbb{Z}$ &  $0, \frac{1}{2}, 0$   & $0$ \vspace{1mm}\\
\vspace{1mm} &$\mathbb{Z}/2\mathbb{Z}$ &  $\frac{1}{2}, \frac{1}{2}, 0 $   & $1$ \vspace{1mm}\\
\vspace{1mm}$7$ &$ \mathbb{Z}/12\mathbb{Z}$ & $\frac{1}{6}, \frac{11}{12}, \frac{5}{12}$  & $\frac{23}{12}$ \vspace{1mm}\\
\vspace{1mm}$8$ &  $ \mathbb{Z}/20\mathbb{Z}$ & $\frac{1}{10},\frac{7}{20},\frac{19}{20}$  & $\frac{39}{20}$ \vspace{1mm}\\
\vspace{1mm}$9$ & $ \mathbb{Z}/16\mathbb{Z}$ & $\frac{3}{8},\frac{1}{8},\frac{15}{16}$ &  $\frac{31}{16}$ \vspace{1mm}\\
\vspace{1mm}$10$ & $ \mathbb{Z}/14\mathbb{Z}$ &  $\frac{2}{7},\frac{5}{14},\frac{3}{14}$ & $ \frac{3}{14}$ \vspace{1mm}\\
\vspace{1mm}$11$ & $ \mathbb{Z}/12\mathbb{Z}$ & $\frac{7}{12},\frac{11}{12},\frac{1}{4} $ & $\frac{19}{12}$ \vspace{1mm}\\
\vspace{1mm}$12$ & $ \mathbb{Z}/18\mathbb{Z}$ &  $\frac{13}{18},\frac{8}{9},\frac{1}{3} $ & $\frac{31}{18}$ \vspace{1mm}\\
\vspace{1mm}$13$ & $ \mathbb{Z}/14\mathbb{Z}$ & $\frac{2}{7},\frac{5}{7},\frac{3}{14},\frac{1}{14}$  & $\frac{12}{7}$ \vspace{1mm}\\
\vspace{1mm}  & $ \mathbb{Z}/2\mathbb{Z}$ &  $0,\frac{1}{2},0,\frac{1}{2}$  & $0$ \vspace{1mm}\\
\vspace{1mm}$14$ &$ \mathbb{Z}/23\mathbb{Z}$ &  $\frac{9}{23},\frac{17}{23},\frac{5}{23},\frac{1}{23}$ & $\frac{40}{23}$  \vspace{1mm}\\
\vspace{1mm}$15$ & $ \mathbb{Z}/31\mathbb{Z}$ & $  \frac{13}{31},\frac{5}{31},\frac{12}{31},\frac{14}{31} $ & $\frac{44}{31}$ \vspace{1mm}\\
 \vspace{1mm}$16$ &  $ \mathbb{Z}/10\mathbb{Z}$ & $\frac{3}{10},\frac{1}{5},\frac{7}{10},\frac{1}{10} $  & $\frac{17}{10}$ \vspace{1mm}\\
 \vspace{1mm} & $ \mathbb{Z}/2\mathbb{Z}$ &  $\frac{1}{2},\frac{1}{2},0,\frac{1}{2} $ & $\frac{1}{2}$\\
  \vspace{1mm}$17$ & $ \mathbb{Z}/12\mathbb{Z}$ & $\frac{1}{6},\frac{5}{6},\frac{1}{3},\frac{5}{12},\frac{5}{12} $  & $\frac{17}{12}$ \vspace{1mm}\\
   \vspace{1mm} & $ \mathbb{Z}/2\mathbb{Z}$ & $0,0,\frac{1}{2},0,\frac{1}{2} $  & $0$ \vspace{1mm}\\
 \vspace{1mm}  &$ \mathbb{Z}/2\mathbb{Z}$ & $0,\frac{1}{2},0,0,\frac{1}{2}$  & $1$ \vspace{1mm}\\
  $18$ &$ \mathbb{Z}/44\mathbb{Z}$ & $\frac{13}{44},\frac{5}{44},\frac{27}{44},\frac{31}{44},\frac{25}{44}$ & $\frac{57}{44}$ \vspace{1mm} \\
\hline\\
\end{longtable}

Also,
$$\beta = \sum_{k=1}^{\ell_k-3} r_k v_k $$
is defined by the coefficients in Table A.2.2.

\begin{longtable}{llll}
\caption*{Table A.2.2: Coefficients  $s_1,\ldots, s_{23-\ell_k}$ and discriminants $q_{E_k}(\beta)$}
\vspace{-5.5mm}\\
\\
\hline
  $k$ &Cyclic groups & Coefficients  $s_1,\ldots, s_{23-\ell_k}$  & $q_{E_k} (\beta)$    \vspace{1mm}
  %\vspace{1mm}
    \\
  \hline
  %\vspace{.1mm}
  \endhead
\vspace{1mm}$6$  &$ \mathbb{Z}/4\mathbb{Z}$ & $\frac{1}{2},\frac{1}{4},\frac{1}{2},\frac{1}{4},0,\frac{1}{2},0,\frac{1}{4},\frac{1}{2},\frac{3}{4},\frac{3}{4},\frac{1}{2},\frac{3}{4},0,0,0,0 $   & $\frac{1}{4}$ \vspace{1mm}\\
\vspace{1mm}  &$ \mathbb{Z}/2\mathbb{Z}$&  $0,0,0,0,0,0,0,0,0,0,\frac{1}{2},0,\frac{1}{2},0,\frac{1}{2},0,\frac{1}{2} $   & $0$ \vspace{1mm}\\
\vspace{1mm} &$ \mathbb{Z}/2\mathbb{Z}$ & $0,\frac{1}{2},\frac{1}{2},0,0,\frac{1}{2},0,\frac{1}{2},0,\frac{1}{2},\frac{1}{2},0,0,0,0,0,0  $   & $1$ \vspace{1mm}\\
\vspace{1mm}$7$  &$ \mathbb{Z}/12\mathbb{Z}$ & $\frac{2}{3}, \frac{1}{3}, \frac{2}{3}, \frac{5}{6}, \frac{2}{3}, \frac{1}{2}, \frac{1}{3}, \frac{1}{6}, \frac{1}{12}, \frac{11}{12}, \frac{11}{12}, \frac{5}{6}, \frac{3}{4}, 0, \frac{1}{4}, \frac{1}{2}, \frac{3}{4}$ & $\frac{1}{12}$  \vspace{1mm}\\
\vspace{1mm}$8$ &$ \mathbb{Z}/20\mathbb{Z}$ & $\frac {1} {5}, \frac {3} {5}, \frac {2} {5}, \frac {7} {10}, \frac {2} {5}, \frac {1} {10}, \frac {1} {20}, \frac {3} {4}, \frac {3}{5}, \frac {4} {5}, \frac {3} {4}, \frac {1} {2}, \frac {1} {4},\frac {3} {5}, \frac {19} {20}, \frac {3} {10}, \frac {13} {20}$&  $\frac{1}{20}$ \vspace{1mm}\\
\vspace{1mm}$9$ &$ \mathbb{Z}/16\mathbb{Z}$ & $\frac{3}{4},\frac{3}{8},\frac{9}{16},\frac{1}{16},\frac{1}{2},\frac{7}{16},\frac{3}{8},\frac{1}{4},\frac{1}{8},\frac{5}{8},\frac{1}{4},\frac{7}{8},\frac{15}{16},0,0,0,0$ & $\frac{1}{16}$ \vspace{1mm}\\
\vspace{1mm}$10$  &$ \mathbb{Z}/14\mathbb{Z}$ & $ \frac{5}{7},\frac{5}{14},\frac{9}{14},\frac{11}{14},\frac{4}{7},\frac{5}{7},\frac{6}{7},\frac{1}{2},0,\frac{1}{2},0,\frac{6}{7},\frac{5}{7},\frac{4}{7},\frac{3}{7},\frac{2}{7},\frac{1}{7} $ & $\frac{25}{14} $  \vspace{1mm}\\
\vspace{1mm}$11$ &$ \mathbb{Z}/12\mathbb{Z}$ & $ \frac{1}{2},\frac{3}{4},\frac{1}{4},\frac{1}{4},\frac{1}{2},\frac{5}{12},\frac{1}{3},\frac{1}{3},\frac{1}{6},\frac{11}{12},\frac{5}{6},\frac{3}{4},\frac{2}{3},\frac{1}{3},0,\frac{2}{3},\frac{1}{3} $ & $\frac{5}{12}$ \vspace{1mm}\\
\vspace{1mm}$12$ &$ \mathbb{Z}/18\mathbb{Z}$& $\frac{1}{3},\frac{2}{3},\frac{1}{3},\frac{1}{18},\frac{1}{9},\frac{5}{6},\frac{5}{9},\frac{5}{18},\frac{2}{9},\frac{1}{9},\frac{1}{3},\frac{2}{3},0,0,0,0,0 $ & $ \frac{5}{18}$ \vspace{1mm}\\
\vspace{1mm}$13$ &$ \mathbb{Z}/14\mathbb{Z}$& $\frac{3}{7},\frac{5}{7},\frac{13}{14},\frac{5}{14},\frac{2}{7},\frac{9}{14},0,0,\frac{1}{7},0,0,\frac{1}{14},\frac{1}{7},\frac{6}{7},\frac{4}{7},\frac{2}{7}$   & $\frac{2}{7}$   \vspace{1mm}\\
\vspace{1mm} &$ \mathbb{Z}/2\mathbb{Z}$ & $0,0,0,0,0,0,0,0,0,\frac{1}{2},0,\frac{1}{2},0,\frac{1}{2},0,\frac{1}{2} $  & $0$ \vspace{1mm}\\
\vspace{1mm}$14$ &$ \mathbb{Z}/23\mathbb{Z}$ & $\frac{6}{23},\frac{3}{23},\frac{20}{23},\frac{18}{23},\frac{13}{23},\frac{8}{23},\frac{6}{23},\frac{4}{23},\frac{2}{23},\frac{1}{23},\frac{2}{23},\frac{3}{23},\frac{7}{23},\frac{11}{23},\frac{15}{23},\frac{19}{23} $ & $\frac{6}{23}$ \vspace{1mm}\\
\vspace{1mm}$15$ &$ \mathbb{Z}/31\mathbb{Z}$ & $ \frac{28}{31},\frac{14}{31},\frac{17}{31},\frac{4}{31},\frac{8}{31},\frac{26}{31},\frac{13}{31},\frac{1}{31},\frac{16}{31},\frac{30}{31},\frac{29}{31},\frac{28}{31},\frac{10}{31},\frac{23}{31},\frac{5}{31},\frac{18}{31}$ & $\frac{18}{31}$ \vspace{1mm}\\
 \vspace{1mm}$16$ &$ \mathbb{Z}/10\mathbb{Z}$& $0,\frac{1}{2},\frac{1}{2},\frac{3}{10},\frac{3}{5},\frac{2}{5},\frac{1}{5},\frac{1}{10},\frac{1}{5},\frac{3}{10},\frac{2}{5},0,\frac{3}{5},\frac{1}{5},\frac{4}{5},\frac{2}{5}$  & $\frac{3}{10}$ \vspace{1mm}\\
 \vspace{1mm}  &$ \mathbb{Z}/2\mathbb{Z}$ & $0,0,0,0,0,0,0,\frac{1}{2},0,\frac{1}{2},0,\frac{1}{2},0,\frac{1}{2},0,\frac{1}{2} $ & $\frac{3}{2}$ \\
  \vspace{1mm}$17$ &$ \mathbb{Z}/12\mathbb{Z}$ & $\frac{1}{3},\frac{1}{6},\frac{1}{4},\frac{7}{12},\frac{5}{6},\frac{5}{12},0,0,\frac{5}{12},\frac{5}{12},\frac{1}{2},\frac{3}{4},0,\frac{2}{3},\frac{1}{3}$  & $\frac{7}{12}$ \vspace{1mm}\\
   \vspace{1mm}  &$ \mathbb{Z}/2\mathbb{Z}$& $0, 0, 0, 0, 0, 0, 0, 0, 0, \frac{1}{2}, \frac{1}{2}, 0, \frac{1}{2}, 0, \frac{1}{2} $  & $0$ \vspace{1mm}\\
    \vspace{1mm} &$ \mathbb{Z}/2\mathbb{Z}$& $0,0,0,0,0,0,0,0,\frac{1}{2},\frac{1}{2},0,0,0,0,0 $  & $1$ \vspace{1mm}\\
  $18$ &$ \mathbb{Z}/44\mathbb{Z}$ &  $\frac{1}{22},\frac{1}{44},\frac{43}{44},\frac{3}{11},\frac{6}{11},\frac{9}{11},\frac{5}{44},\frac{9}{22},\frac{31}{44},\frac{7}{22},\frac{29}{44},\frac{17}{44},\frac{17}{22},\frac{2}{11},\frac{13}{22}$ & $\frac{31}{44}$ \vspace{1mm}\\
\hline\\
\end{longtable}

\subsection*{A.3 Precise data for the proof of Theorem 5.1}

In this write-up, we show detailed data for $k\in \{6,9,13, 16, 17\}$ for the proof of Theorem 5.1 of [MN]  in Section 5.1.
For the detailed notations, see [MN].

\subsubsection*{\underline{Case of $k=6$}}

$T_6$ is $\langle F,(O),a_1,\dots,a_7,b_1,\dots,b_7\rangle_\mathbb{Z} $.
Since $\mathrm{det}(T_6)=-64$ and we have a section $O'$ with $2O'=O$, it follows $\left[\widehat{T}_6:T_6\right]=2 \text{ or } 4 \text{ or } 8$.

Suppose
$\left[\widehat{T}_6:T_6\right]=4$.
Then, there exists a section $R_1$ such that $4R_1 =O$.
$R_1$ must satisfy
\begin{align*}
(a_p\cdot (R_1)) = ( b_q\cdot (R_1))  =1, \quad \quad
p,q\in\{0,2,6\}
\end{align*}
except for $(p,q)=(0,0)$.
Set $T_6'=\langle T_6,(R_1)\rangle_\mathbb{Z} $ and 
put $((O)\cdot (R_1))=\kappa\in \mathbb{Z}_{\geq 0}.$
Then, 
\begin{align}
\mathrm{det}(T_6')=
\begin{cases}
32(5+4\kappa) & p=0 \text{ or } q=0, \\ 
64(1+2\kappa)  & \text{otherwise}.\\
\end{cases}
\end{align}
This is a contradiction.

Suppose
$\left[\widehat{T}_6:T_6\right]=8$.
Then, there exists a section $R_2$ with $8R_2 =O$.
This section must satisfy
$(a_p\cdot (R_1)) =( b_q\cdot (R_1)) =1$
where $p$ or $q$ is a odd number.
Set
$T_6''=\langle T_6, (R_2) \rangle_\mathbb{Z} $
and put $((O)\cdot (R_2))=\kappa\in \mathbb{Z}_{\geq 0}$.
Then, we calculate
\begin{align*}
{\rm det}(T_6'')=
\begin{cases}
 8(25+16\kappa) & (p,q)=(0,1),(1,0),(0,7),(7,0),\\ 
8(17+16\kappa) & (p,q)=(0,3),(3,0),(0,5),(5,0),\\
16(9+8\kappa) & (p,q)=(1,1),(1,7),(7,1),(7,7),\\
8(13+16\kappa) & (p,q)=(1,2),(2,1),(1,6),(6,1),(7,2),(2,7),(7,6),(6,7),\\
16(5+8\kappa)& (p,q)=(1,3),(3,1),(1,5),(5,1),(7,3),(3,7),(7,5),(5,7),\\
8(9+16\kappa)& (p,q)=(1,4),(4,1),(7,4),(4,7),\\
8(5+16\kappa)& (p,q)=(2,3),(3,2),(2,5),(5,2),(6,3),(3,6),(6,5),(5,6) ,\\
16(1+8\kappa)& (p,q)=(3,3),(3,5),(5,3),(5,5) ,\\
8(1+16\kappa)& (p,q)=(3,4),(4,3),(5,4),(4,5) .\\
\end{cases}
\end{align*}
This is a contradiction.

Hence, we have
$\left[\widehat{T}_6:T_6\right]=2$.

We set $\overline{T}_6=\langle F,(O),(O'),a_2,\dots,a_7,b_1,\dots,b_7\rangle_\mathbb{Z} $.
Then, $\mathrm{det}\left(\overline{T}_6\right)=-16$.
So,
we obtain $\overline{T}_6=\widehat{T}_6$.

Set
$E_6=\langle F,(O),(O'),(Q),a_2,\dots,a_7,b_1,\dots,b_7\rangle_\mathbb{Z} =\left\langle \overline{T}_6,(Q) \right\rangle_\mathbb{Z} $.
Since $\mathrm{det}(E_6)=16$,
$[\mathrm{NS}(S_6):E_6]=1 \text{ or } 2 \text{ or } 4$.

If we assume
$[\mathrm{NS}(S_6):E_6]=2$,
we have a section $Q_1'$ with $2Q'_1=Q$.
Such a $Q'_1$ satisfies
\begin{align*}
(a_p\cdot (Q'_1)) = 1, \quad  (b_q \cdot (Q'_1))  = 1, 
\quad\quad p,q \in \{ 1,5 \}.
\end{align*}
Set
$\widetilde{E}_6=\langle E_6,(Q'_1)\rangle_\mathbb{Z} $.
Then, we have
\begin{align*}
\mathrm{det}\left(\widetilde{E}_6\right)=
\begin{cases}
92   & (p,q)=(1,1),(5,5),\\ 
-16   & (p,q)=(1,5),(5,1).\\
\end{cases}
\end{align*}
This is a contradiction.

If we assume $\left[\widehat{T}_6:T_6\right]=4$,
there exists a section $Q_2'$ with $4Q'_2=Q$.
However, due to the group structure of $I_8$,
such a section cannot exist.

Hence, we have $[\mathrm{NS}(S_6):E_6]=1$.

\subsubsection*{\underline{Case   of  $k=9$}}

$T_9$ is $\langle F,(O),a_1,\dots,a_5,b_1,\dots,b_9\rangle_\mathbb{Z} $.
Set $E'_9=\langle T,(Q) \rangle_\mathbb{Z}$.
Since $\mathrm{det}(E'_9)=64$ and we have the section $O'$ with $2O'=O$,
$[\mathrm{NS}(S_9):E'_9]=2 \text{ or } 4 \text{ or } 8$.
Since
$\mathrm{det}(T_9)=-60$,
we have $\left[\widehat{T}_9:T_9\right]=2$.

Set $\overline{T}_9=\langle F,(O),(O'),a_1,\dots,a_5,b_1,\dots,b_8\rangle_\mathbb{Z} .$
Then, $|\mathrm{det}\left(\overline{T}_9\right)|=15=|\mathrm{det}\left(\widehat{T}_9\right)|$.
Hence, $\overline{T}_9=\widehat{T}_9$.

Set $E_9=\langle F,(O),(O'),(Q),a_1,\dots,a_5,b_1,\dots,b_8\rangle_\mathbb{Z} =\left\langle \overline{T}_9,(Q) \right\rangle_\mathbb{Z} $.
Since $\mathrm{det}(E_9)=16$,
$[\mathrm{NS}(S_9):E_9]=1 \text{ or } 2 \text{ or } 4.$
However, since ${{\rm NS} (S_9)} $ is an even lattice of rank $17$,
$[\mathrm{NS}(S_9):E_9]\not= 4.$

By observing the group structure of $I_6$ and $I_{10}$,
a section with  $Q'$ with $2Q'=Q$  cannot exist.

Suppose there exists a section $Q'$ with $2Q'=Q+O'$.
Set $\check{E}_9 = \langle (Q'), \overline{T}_9 \rangle_\mathbb{Z}$.
We have
\begin{align*}
|{\rm det} \left( \check{E}_9 \right)|
=
\begin{cases}
 26 \quad &( a_2 \cdot (Q')) = (b_4 \cdot (Q')) =1, \\
 14  \quad &( a_2 \cdot (Q')) = ( b_9  \cdot (Q')) =1,\\
 10 \quad &( a_5 \cdot (Q')) = ( b_4 \cdot (Q')) =1,\\
 74 \quad &( a_5 \cdot (Q')) = ( b_9 \cdot (Q')) =1.
\end{cases}
\end{align*}
This implies that such a section $Q'$ does not exist.

Hence, $[\mathrm{NS}(S_9):E_9]$ is not equal to $2$.
Therefore, $[\mathrm{NS}(S_9):E_9]=1$.

\subsubsection*{\underline{Case   of  $k=13$}}

$T_{13}$ is $\langle F,(O),a_1,\dots,a_5,b_1,c_1,\dots,c_7\rangle_\mathbb{Z} $.
Set $E'_{13}=\langle T,(Q) \rangle_\mathbb{Z}$.
Since
$\mathrm{det}(E'_{13})=-112$ and we have a section $O'$ with $2O'=O$,
$[\mathrm{NS}(S_{13}):E'_{13}]=2 \text{ or } 4$.

If we assume $\left[\widehat{T}_{13}:T_{13}\right]=4$,
there exists a section $R$ with $4R=O$.
Then,
\begin{align*}
(a_p\cdot(R) ) = (b_q\cdot(R)) = ( c_r\cdot(R))=1,\quad\quad
 p\in\{0,3\},q\in\{0,1\},r\in\{2,6\}
\end{align*}
Set
$T_{13}'=\langle T_{13},(R)\rangle_\mathbb{Z} $.
Put $((O)\cdot (R) =\kappa \in \mathbb{Z}_{\geq 0}$.
The discriminant of $T_{13}'$ is
\begin{align*}
\mathrm{det}(T_{13}')=
\begin{cases}
-48(5+4\kappa) & (p,q,r)=(0,0,2),(0,0,6),\\
-192(1+\kappa) & (p,q,r)=(0,1,2),(0,1,6),\\ 
-96(1+2\kappa) & (p,q,r)=(3,0,2),(3,0,6),\\
-48(1+4\kappa) & (p,q,r)=(3,1,2),(3,1,6).\\
\end{cases}
\end{align*}
It leads to a contradiction.

Setting $\overline{T}_{13}=\langle F,(O),(O'),a_1,\dots,a_4,b_1,c_1,\dots,c_7\rangle_\mathbb{Z} $
we have
$\mathrm{det}\left(\overline{T}_{13}\right)=24$.
Since $\mathrm{det}(T_{13})=96$,
we have $\overline{T}_{13}=\widehat{T}_{13}$.

Set
$E_{13}=\langle F,(O),(O'),(Q),a_1,\dots,a_4,b_1,c_1,\dots,c_7\rangle_\mathbb{Z} =\left\langle \overline{T}_{13},(Q) \right\rangle_\mathbb{Z} $.
We have
We have $\mathrm{det}(E_{13})=28$.
So, $[\mathrm{NS}(S_{13}):E_{13}]=1 \text{ or } 2.$

By observing the group structure of $I_6,I_2$ and $I_8$,
a section with  $Q'$ with $2Q'=Q$ cannot exist.

Suppose there exists a section $Q'$ with $2Q'=Q+O'$.
Then, there also exist sections $Q'+O'$ and $-Q'$.
Taking $Q_0' \in \{Q', Q'+O', -Q'\}$,
let us consider the group $\check{E}_{13}=\langle (Q_0'), \overline{T}_{13} \rangle_\mathbb{Z}$.
Since $\check{E}_{13}$ should give an overlattice of $E_{13}$,
the following calculation  for  $Q_0'$  leads a contradiction :
\begin{align*}
|{\rm det} \left( \check{E}_{13} \right)|
=
\begin{cases}
 15 \quad &(p,q,r)=(2,0,3), \\
 9  \quad &(p,q,r)=(4,0,1),\\
 41 \quad &( p,q,r) = (1,0,1),\\
 31 \quad &( p,q,r)=(5,0,3),
\end{cases}
\end{align*}
where $(p,q,r)$ means $(a_p \cdot (Q_0')) = (b_q \cdot (Q_0')) =(c_r \cdot (Q_0'))=1. $
Hence, such a section $Q'$ does not exist.

Therefore, $[\mathrm{NS}(S_{13}):E_{13}]=1$.

\subsubsection*{\underline{Case   of  $k=16$}}

$T_{16}$ is $\langle F,(O),a_1,\dots,a_4,b_1,\dots,b_9\rangle_\mathbb{Z} $.
Set $E_{16}=\langle T,(Q) \rangle_\mathbb{Z}$.
Since $\mathrm{det}(E_{16})=-20$, $[\mathrm{NS}(S_{16}):E_{16}]=1 \text{ or } 2$.

Suppose $\left[\widehat{T}_{16}:T_{16}\right]=2$.
Then, there exists a section $R$ with $2R=O$.
So, $R$ should intersect as
$(a_0\cdot(R) ) = ( b_5\cdot(R) ) =1.$
The discriminant of $T_{16}'=\langle T_{16},R\rangle_\mathbb{Z}$ is
$
\mathrm{det}(T_{16}')=-25(3+\kappa),
$
where $\kappa=((O) \cdot (R))\in \mathbb{Z}_{\geq 0}.$
Therefore,
$\left[\widehat{T}_{16}:T_{16}\right]=1$.

Suppose 
there exists a section $Q'$ with $2Q'=Q$.
So, $Q'$ satisfies
$
(a_1\cdot(Q')) = (b_2\cdot(Q')) =1 \text{ or }
(a_1\cdot(Q')) = ( b_4\cdot(Q')) =1.
$
The discriminant of $\widetilde{E}_{16}=\langle E_{16},(Q')\rangle_\mathbb{Z}$
is
\begin{align*}
\Big|\mathrm{det}\left(\widetilde{E}_{16}\right)\Big|=
\begin{cases}
30 & (a_1\cdot(Q')) = (b_2\cdot(Q')) =1,\\
34 & (a_1\cdot(Q')) = ( b_4\cdot(Q')) =1.
\end{cases}
\end{align*}
Thus, we have $[\mathrm{NS}(S_{16}):E_{16}]=1$.

\subsubsection*{\underline{Case   of  $k=17$}}

$T_{17}$ is $\langle F,(O),a_1,\dots,a_5,b_1,c_1,d_1,\dots,d_5\rangle_\mathbb{Z} $.
Set $E'_{17}=\langle T,(Q) \rangle_\mathbb{Z}$.
Since
$\mathrm{det}(E'_{17})=192$ and we have the section $O'$ with $2O'=O$,
$[\mathrm{NS}(S_{17}):E'_{17}]=2 \text{ or } 4 \text{ or } 8$.

Suppose
$\left[\widehat{T}_{17}:T_{17}\right]=4 \text{ or } 8$.
Then, there exists a section $R$ with $4R=O$ or $8R=O$.
However, due to the structure of $I_6$ and $I_2$,
such a section cannot exist.
Therefore, $\left[\widehat{T}_{17}:T_{17}\right]=2$.

Set $\overline{T}_{17}=\langle F,(O),(O'),a_1,\dots,a_4,b_1,c_1,d_1\dots,d_5\rangle_\mathbb{Z} $.
We have $\mathrm{det}\left(\overline{T}_{17}\right)=-36$.
Together with $\mathrm{det}(T_{17})=-144$,
it follows $|\mathrm{det}\left(\overline{T}_{17}\right)|=36$ and
$\overline{T}_{17}=\widehat{T}_{17}$.

Set
$E_{17}=\langle F,(O),(O'),(Q),a_1,\dots,a_4,b_1,c_1,d_1\dots,d_5\rangle_\mathbb{Z} =\left\langle \overline{T}_{17},(Q) \right\rangle_\mathbb{Z} $.
Since
$\mathrm{det}(E_{17})=48$,
$[\mathrm{NS}(S_{17}):E_{17}]=1 \text{ or } 2 \text{ or } 4$.
However, since ${{\rm NS} (S_{17})} $ is an even lattice of rank $15$,
$[\mathrm{NS}(S_9):E_9]\not= 4.$

By observing the group structure of $I_6$ and $I_2$,
a section with  $Q'$ with $2Q'=Q$ cannot exist.

Suppose there exists a section $Q'$ with $2Q'=Q+O'$.
It implies the existence of the sections $Q'+O'$ and $-Q'$.
For $Q_0' \in \{Q', Q'+O', -Q'\}$,
let us consider the group $\check{E}_{17}=\langle (Q_0'), \overline{T}_{17} \rangle_\mathbb{Z}$.
According to the following  calculation of ${\rm det} \left( \check{E}_{17} \right)$,
one can observe that  either
\begin{itemize}
\item $|{\rm det}\left(\check{E}_{17} \right)| \geq 48$ for every $Q_0' \in \{Q', Q'+O', -Q'\}$   
\end{itemize}
or
\begin{itemize}
\item there exist a certain  $Q_0' \in \{Q', Q'+O', -Q'\}$ such that $\check{E}_{17}$, which is determined by $Q_0'$,  cannot be an overlattice of $ E_{17}$
\end{itemize}
holds for each candidate $Q'$ satisfying $2Q'=Q+O'$:
\begin{align*}
|{\rm det} \left( \check{E}_{17} \right)|
=
\begin{cases}
 24 \quad &(p,q,r,s)=(5,0,0,2), (2,0,1,2), \\
 156  \quad &(p,q,r,s)=( 5,0,0,5),\\
 60 \quad & (p,q,r,s) =(5,1,1,2),\\
 48 \quad &( p,q,r,s) = (2,0,0,2), (2,0,1,5),(5,0,1,2),(5,1,1,5),\\
\end{cases}
\end{align*}
where $(p,q,r,s)$ means $(a_p \cdot (Q_0')) = (b_q \cdot (Q_0')) =(c_r \cdot (Q_0'))=(d_s\cdot (Q_0'))=1. $
Thus, we can see the nonexistence of such a section $Q'$.

Hence, $[\mathrm{NS}(S_{17}):E_{17}]\not=2$ and we have $[\mathrm{NS}(S_{17}):E_{17}] =1$.

\section*{Acknowledgment}
The authors sincerely appreciate the anonymous reviewer's  valuable comments.
They are helpful for the authors to correct the mathematical  argument and  improve the manuscript.
This second author is supported by JSPS Grant-in-Aid for Scientific Research (22K03226) and JST FOREST Program (JPMJFR2235).

\begin{center}
\hspace{7.7cm}\textit{Tomonao  Matsumura}\\
\hspace{7.7cm}\textit{Core Concept Technologies Inc.,}\\
\hspace{7.7cm}\textit{11F DaiyaGate Ikebukuro, }\\
\hspace{7.7cm}\textit{1-16-15 Minamiikebukuro, Toshima-ku, Tokyo}\\
\hspace{7.7cm}\textit{171-0022, Japan}\\
 \hspace{7.7cm}\textit{(E-mail: tomonaomatsumura@gmail.com)}
  \end{center}

\begin{center}
\hspace{7.7cm}\textit{Atsuhira  Nagano}\\
\hspace{7.7cm}\textit{Faculty of Mathematics and Physics,}\\
 \hspace{7.7cm} \textit{Institute of Science and Engineering,}\\
\hspace{7.7cm}\textit{Kanazawa University,}\\
\hspace{7.7cm}\textit{Kakuma, Kanazawa, Ishikawa}\\
\hspace{7.7cm}\textit{920-1192, Japan}\\
 \hspace{7.7cm}\textit{(E-mail: atsuhira.nagano@gmail.com)}
  \end{center}

\end{document}